\documentclass[11pt]{article}

\usepackage[title]{appendix}

\usepackage{xcolor}

\usepackage{pst-all}
\usepackage{pgfplots}
\usepackage{pst-plot}
\usetikzlibrary{calc}

\usepackage{xcolor}
\usepackage{tikz}
\usepackage{tkz-euclide}

\usepackage{bbm}
\usepackage{hyperref}
\hypersetup{
    colorlinks=true,              
    breaklinks=true,              
    urlcolor= black,              
    linkcolor= blue,              
    bookmarksopen=false,
    filecolor=black,
    citecolor=blue,
    linkbordercolor=blue}

\usepackage{chngcntr}

\usepackage{mathrsfs}

\usepackage[margin=1in]{geometry}

\usepackage{amsmath,amsthm,amssymb}

\usepackage{bbm}

\usepackage{graphicx}
\usepackage{faktor}

\newcommand{\R}{\mathbb{R}}

\newcommand{\N}{\mathbb{N}}

\newcommand{\Z}{\mathbb{Z}}

\newcommand{\SL}{\text{SL}}

\newcommand{\SO}{\text{SO}}

\newcommand{\bb}{\normalfont \textbf{b}}
\newcommand{\cc}{\normalfont \textbf{c}}

\newcommand{\e}{\epsilon}

\newcommand{\de}{\delta}

\newcommand{\la}{\lambda}

\newcommand{\La}{\Lambda}

\newcommand{\ga}{\gamma}
\newcommand{\Ga}{\Gamma}

\newtheorem{theorem}{Theorem}[section]

\newtheorem{lemma}[theorem]{Lemma}

\newtheorem{proposition}[theorem]{Proposition}

\newtheorem{corollary}[theorem]{Corollary}

\newtheorem{claim}{Claim}

\newtheorem{tad}[theorem]{Theorem and Definition}

\theoremstyle{definition}

\newtheorem{definition}[theorem]{Definition}

\newtheorem{example}[theorem]{Example}

\theoremstyle{remark}

\newtheorem{remark}[theorem]{Remark}

\numberwithin{equation}{section}

\usepackage{setspace}

\usepackage[
style=alphabetic,
sorting=nty,
maxnames=99,
maxalphanames=99
]{biblatex}
\renewbibmacro{in:}{}
\DeclareFieldFormat[article]{citetitle}{\mkbibemph{#1}}
\DeclareFieldFormat[article]{title}{\mkbibemph{#1}}
\DeclareFieldFormat[article]{journaltitle}{#1}

\usepackage{enumerate}

\AtEndDocument{%
  \par
  \medskip
  \begin{tabular}{@{}l@{}}%
    \textsc{Department of mathematics, the Ohio State University}\\
    \textit{E-mail address}: \texttt{xing.211@osu.edu}
  \end{tabular}}
  
\title{Some measure-theoretical and dynamical properties of successive minima on the space of unimodular lattices}
\author{Hao Xing}

\begin{document}
\date{}
\maketitle

\begin{abstract}
In this article, we study the relation between lattice basis and successive minima and give an estimate for the measure-theoretical distribution of successive minima. As consequences, we also discuss some logarithm laws associated to higher sucessive minima.
\end{abstract}

\section{Some properties of successive minima}

\newcommand{\pl}{$\pi_1(\Lambda)$}

\subsection{Lattice basis and successive minima}

Recall that for a positive integer $d$ and a lattice $\Lambda \subset \mathbb R^d$, and for each $j = 1,\dots, d$, The $j$-th minimum of a lattice $\Lambda \subset \mathbb R^d$, denoted $\lambda_j (\Lambda)$, is the
infimum of $\lambda$ such that the set $\{r \in \Lambda : \|r\| \le  \lambda \}$ contains $j$ linearly independent vectors. (with respect to the $l^2$ norm on $\mathbb R^d$).

A natural question is, can the successive minima always attained by a basis of the rank $d$ lattice $\Lambda$? In other words, does there exist a basis $\{v_1,\dots,v_d\}$ of $\Lambda$ such that 
\begin{equation*}
    \|v_j\|=\lambda_j, ~\text{for}~ j\in \{1,2,\dots,d\}.
\end{equation*}

The answer is positive if and only if $d\le 4$, as are shown in the following theorem and example

\begin{theorem}\label{thm:A.1}
Let $\Lambda$ be a lattice $\mathbb R^d$. Assume that $d\le 4$, then there exist a basis $\{v_1,\dots,v_d\}$ of $\Lambda$ such that 
$$\|v_j\|=\lambda_j, ~\text{for}~ j\in \{1,2,\dots,d\}.$$
\end{theorem}

\vspace{1cm}
The case when $d=1$ is trivial. To prove this theorem for the cases $d=2,3$, we need the following lemma from Euclidean geometry.

\begin{lemma}~\\
(1) The minimal distance from any point in the interior of a parallelogram in $\mathbb R^2$ to its vertices is always strictly less than the maximal length of the edges of the parallelogram.~\\
(2) The minimal distance from a point in the interior of a parallelepiped  in $\mathbb R^3$ to its vertices is always strictly less than the maximal length of three linearly independent vectors form by the vertices , with at least two of them being the edges of the parallelopiped. In particular, these three vector will span the three dimensional lattice spanned by this parallelopiped.
\end{lemma}

\begin{proof}~\\
For the part (1), observe that a parallelogram $ABCB'$ can be divided into two triangles $ABC$ and $AB'C$, and any point $D$ in the interior of $ABCB'$ must fall in either the triangle $ABC$ or the triangle $AB'C$

\iffalse

\begin{tikzpicture}
\draw (-4,0) -- (4,0);
\draw (-4,0) -- (2,3);
\draw (4,0) -- (2,3);

\end{tikzpicture}
\fi

\begin{figure}

\begin{center}

\begin{tikzpicture}[scale=1.2]
\tkzDefPoint(-4,0){A}
\tkzDefPoint(2,3){B}
\tkzDefPoint(-2,-3){B'}
\tkzDefPoint(4,0){C}
\tkzDefPoint(1,1){D}
\tkzDefPoint(0.5,0){E}

\tkzDrawSegment[color=blue, thick](A,B)
\tkzDrawSegment[color=blue, thick](C,B)
\tkzDrawSegment[color=blue, thick](A,B')
\tkzDrawSegment[color=blue, thick](C,B')
\tkzDrawSegment[color=blue, thick](A,C)
\tkzDrawSegment[color=red, thick](B,D)
\tkzDrawSegment[color=red, dotted](E,D)

\tkzDefPointBy[projection=onto A--C](B)
\tkzGetPoint{F}
\tkzDrawSegment[green](B,F)
\tkzMarkRightAngle[,size=0.2,color=green](B,F,C)

\tkzDrawPoints(A,B,B',C,D,E,F)
\tkzLabelPoints[below](C,D,E,F)
\tkzLabelPoints[above, right](B)
\tkzLabelPoints[below](B')
\tkzLabelPoints[below, left](A)

\end{tikzpicture}
\end{center}
\caption{The parallelogram case}
\end{figure}
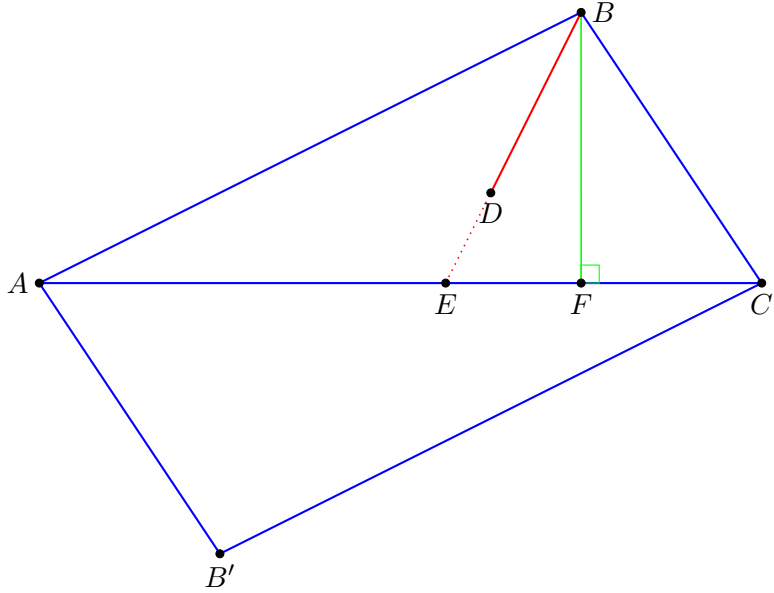

By drawing a line perpendicular to the line $AC$ through the point $B$, we easily see
\begin{equation*}
    |BD|\le |BE| < \max\{|AB|,|BC|\}.
\end{equation*}

For the part (2), first observe that a parallelepiped can be divided into six tetrahedra and any point $x$ in the interior of the parallelepiped, say $ABCDEFGH$ must fall into one of the six.

\newcommand{\Depth}{5}
\newcommand{\Height}{4}
\newcommand{\Width}{6}
\newcommand{\Skew}{1}
\tikzstyle{place}=[circle,draw=black,fill=black,inner sep=1, outer sep=0]

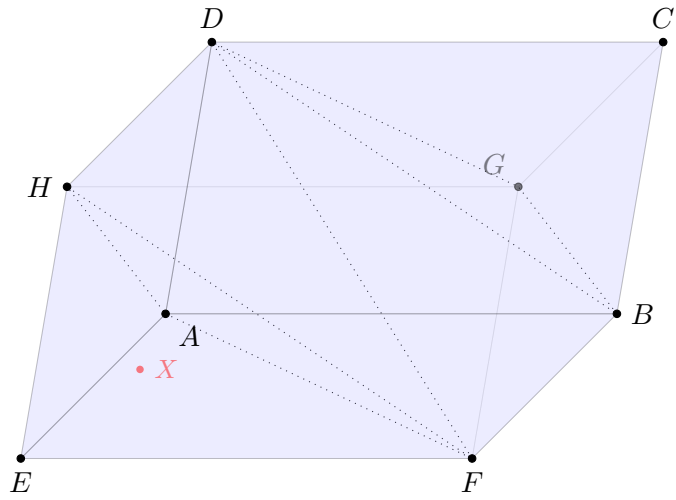
\begin{figure}
\begin{center}

\begin{tikzpicture}[scale=1]

    \coordinate (a) at (0,0,0);
    \coordinate (b) at (\Width,0,0);
    \coordinate (c) at (\Width + \Skew,\Height,\Skew);
    \coordinate (d) at (\Skew,\Height,\Skew);
    \coordinate (e) at (0,0,\Depth);
    \coordinate (f) at (\Width,0,\Depth);
    \coordinate (g) at (\Width + \Skew,\Height,\Depth + \Skew);
    \coordinate (h) at (\Skew,\Height,\Depth + \Skew);
    \coordinate (x) at (0.2*\Width,0.2*\Height,0.8*\Depth);

    \draw [dotted] (a) -- (f) -- (h) -- (a);
    \draw [dotted] (b) -- (g) -- (d) -- (b);
    \draw [dotted] (f) -- (d);
   \draw[fill=blue!30,opacity=0.2] (a) -- (b) -- (c) -- (d) -- (a);
    \draw[fill=blue!30,opacity=0.2] (a) -- (e) -- (f) -- (b) -- (a);
    \draw[fill=blue!30,opacity=0.2] (a) -- (d) -- (h) -- (e) -- (a);
    \draw[fill=blue!30,opacity=0.05] (d) -- (c) -- (g) -- (h) -- (d);
    \draw[fill=blue!30,opacity=0.05] (e) -- (f) -- (g) -- (h) -- (e);
    \draw[fill=blue!30,opacity=0.05] (b) -- (c) -- (g) -- (f) -- (b);

    \node[place, label=below right:{$A$}] at (a) {};
    \node[place, label=right:{$B$}] at (b) {};
    \node[place, label=above:{$C$}] at (c) {};
    \node[place, label=above:{$D$}] at (d) {};
    \node[place, label=below:{$E$}] at (e) {};
    \node[place, label=below:{$F$}] at (f) {};
    \node[place, opacity=0.5, label={above left, opacity=0.5:$G$}] at (g) {};
    \node[place, label=left:{$H$}]  at (h) {};
    \node[circle, fill=red, opacity=0.5, inner sep=1pt, label={right, font=\small, text= red, opacity=0.5:$X$}]  at (x) {};    

\end{tikzpicture}   
\end{center}

\caption{$X$ must fall into one of six tetrahedra.}
\end{figure}

\begin{figure}
\begin{center}
    
\begin{tikzpicture}[scale=1.5]
    \coordinate (a) at (0,0,0);
    \coordinate (b) at (\Width,0,0);
    \coordinate (c) at (\Width + \Skew,\Height,\Skew);
    \coordinate (d) at (\Skew,\Height,\Skew);
    \coordinate (e) at (0,0,\Depth);
    \coordinate (f) at (\Width,0,\Depth);
    \coordinate (g) at (\Width + \Skew,\Height,\Depth + \Skew);
    \coordinate (h) at (\Skew,\Height,\Depth + \Skew);
    \coordinate (x) at (0.2*\Width,0.2*\Height,0.8*\Depth);
    \coordinate (y) at (0.296*\Width,0.7,\Depth);
    \coordinate (z) at (0.281*\Width,0,\Depth);

    \draw[fill=blue!30, opacity=0.3] (a) -- (f) -- (h) -- (a);
    \draw[fill=blue!30, opacity=0.3] (a) -- (e) -- (f) -- (a);
    \draw[fill=blue!30, opacity=0.3] (a) -- (e) -- (h) -- (a);
    \draw[fill=blue!30, opacity=0.1] (e) -- (f)  -- (h) -- (e);
    \draw[dotted] (h) -- (z);

    \node[place, label=below right:{A}] at (a) {};
    %\node[place, label=right:{B}] at (b) {};
    %\node[place, label=above:{C}] at (c) {};
    %\node[place, label=above:{D}] at (d) {};
    \node[place, label=below:{$H$}] at (e) {};
    \node[place, label=below:{$F$}] at (f) {};
    %\node[place, opacity=0.5, label={above left, opacity=0.5:G}] at (g) {};
    \node[place, label=left:{$E$}]  at (h) {};
    \node[circle, fill=red, opacity=0.5, inner sep=1pt, label={left, font=\small, text= red, opacity=0.5:$X$}]  at (x) {};
    \node[circle, fill=red, inner sep=1pt, label={below right, text= red:$Y$}]  at (y) {};
    \node[circle, fill=red, inner sep=1pt, label={below right, text= red:$Z$}]  at (z) {};
    \draw [dotted] (h) -- ($(h)!3.14cm!(x)$);
    \draw (a) -- ($(a)!1.95cm!(y)$);
\end{tikzpicture}
\end{center}
\caption{Construction of points $Y,Z$ in the proof.\label{construction}}
\end{figure}
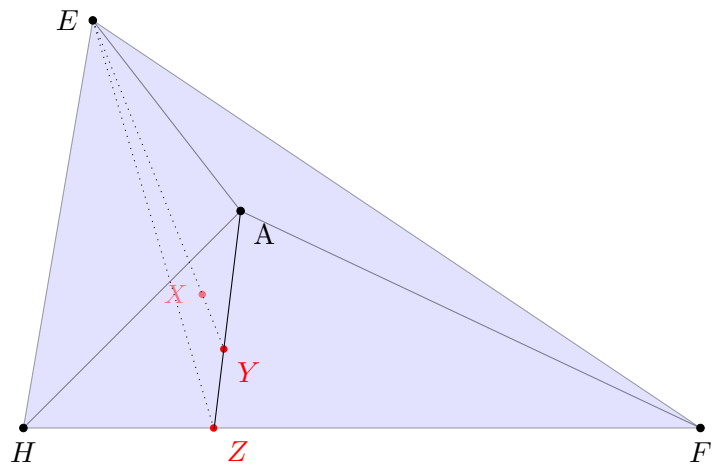

If $X$ falls in the tetrahedra $AFEH$. It follows from the first part of this lemma that 
\begin{equation*}
    |EX|\le |EY| \le \max\{|EA|,|EZ|\} < \max\{|EH|,|EF|,|EA|\}, 
\end{equation*}

If $X$ falls in the tetrahedra $DHAF$. It follows from the first part of this lemma that 
\begin{equation*}
    |DX|\le |DY| \le \max\{|DA|,|DZ|\} < \max\{|DA|,|DH|,|DF|\}, 
\end{equation*}

If $X$ falls in the tetrahedra $BAFD$. It follows from the first part of this lemma that 
\begin{equation*}
    |BX|\le |BY| \le \max\{|BA|,|EZ|\} < \max\{|BA|,|BF|,|BD|\}, 
\end{equation*}

where the construction of auxiliary points and segments as illustrated in the figure \ref{construction} above.

\end{proof}

\begin{proof}[Proof of the theorem for the case d=2,3:]~\\
We do this for $d=3$ and the case $d=2$ is only simpler. Let $v_1, v_2, v_3$ be any linearly independent vectors in $\Lambda$ such that 
\begin{equation*}
    \|v_j\|=\lambda_j, ~\text{for}~ j=\{1,2,3\}.
\end{equation*}

Let $\Lambda_0$ be the lattice spanned by those three vectors. Consider the fundamental domain
$$F:=\{t_1 v_1+t_2 v_2+t_3 v_3: t_i \in [0,1)\}$$
of $\Lambda_0$.

The closure of $F$ is the parallelepipiped spanned by vectors $v_1,v_2,v_3$ at the origin of $\mathbb R^3$ and it follows that $\mathbb R^3=F+\Lambda_0$.

Suppose on the contrary that $\Lambda_0$ is a proper sublattice of $\Lambda$, then there exists a vector $x \in \Lambda - \Lambda_0$. By translating $x$ with a vector in $\lambda_0$, we may assume without loss of generality that 
$$x=t_1 v_1+t_2 v_2+t_3 v_3,$$
where $t_i \in (0,1)$ ($t_i$ cannot be equal to zero since $x \notin \Lambda_0$). So $x$ is in the interior of the parallelepipiped.

By our lemma above, noticing that the length of each edge in the parallelepipiped is equal to $\|v_1\|,\|v_2\|$ or $\|v_3\|$, there exists a vertex $w$ of the parallelepipiped spanned by vectors $v_1,v_2,v_3$ at the origin such that 
$$\|x-w\| < \max\{ \|v_j\|:j=1,2,3\}=\|v_3\|.$$

Translating the vector $x-w$ to the origin. It follows that $x-w, v_1, v_2$ are still linearly independent and this lead to the contradiction to the assumption that $v_3 = \lambda_3(\lambda)$. Therefore we must have 
$$\Lambda_0 = \Lambda,$$
namely $v_1,v_2,v_3$ form a basis of $\lambda$.

\end{proof}

We need the following lemma for the proof of $d=4$ case:
\begin{lemma}
For $k\ge 1$ and $y_i\in \R$ for all $1\le i\le k$, we have the identity
$$S_k:=\sum x_1x_2\cdots x_k = 1,$$
where the sum is over all $2^k$ possible choices of $x_i=y_i$ or $x_i=1-y_i$.
\end{lemma}
\begin{proof}
We perform induction on $k$.

When $k=1$, this sum is simply $1-y_1+y_1=1$.

Assume $S_{k-1}=1$. For general $k$, we observe that $S_k=y_iS_{k-1}+(1-y_i)S_{k-1}=S_{k-1}=1$.
\end{proof}

\begin{proof}[Proof of the theorem for the case d=4:]~\\
Let $v_1, v_2, v_3,v_4$ be any linearly independent vectors in $\Lambda$ such that 
\begin{equation*}
    \|v_j\|=\lambda_j(\Lambda), ~\text{for}~ j=\{1,2,3,4\}.
\end{equation*}
As in the proof of cases $d=2,3$, we let $\Lambda_0$ denote  the lattice spanned by $v_1, v_2, v_3,v_4$.

If $\lambda_0$ is a proper sublattice of $\lambda$, then there exist $v \in \Lambda - \Lambda_0$ such that
$$v=\sum_{i=1}^4 t_i v_i, t_i\in \mathbb R$$

and without loss of generality, we may assume that $t_i\in [0,1)$ for any $i=1,2,3,4$.

Namely $v$ lives in $$\mathscr{P}:=\left \{\sum_{i=1}^4 t_i v_i, t_i\in \mathbb [0,1) \right \}.$$

\vspace{1cm}
\begin{claim}
For any $v_0 \in \Lambda_0$, and any $v\in \La -\Lambda_0$
$$\|v-v_0\| \ge \|v_i\|,~\forall i=1,2,3,4$$
\end{claim}

\begin{proof}[Proof of Claim]
\renewcommand\qedsymbol{\#}

If $\|v-v_0\|<\|v_{i_0}\|$ for some $i_0 \in \{1,2,3,4\}$, then $v_i$, $i\in \{1,2,3,4\} \setminus \{i_0\}$, together with $v-v_0$ would be still linearly independent (since $v\in \La-\Lambda_0$) and form a new system of successive minima with strictly smaller $\lambda_i$.
\end{proof}

\begin{claim} For any $v \in \mathscr P$ we have
\begin{equation*}
    \min_{v_0 \in \Lambda_0} \|v-v_0\|^2 \le \frac{1}{4}\sum_{i=1}^4 \|v_i\|^2.
\end{equation*}
%If the inequality attained, then
%$v=\frac{1}{2}(v_1+v_2+v_3+v_4)$ and $v_i's$ are pairwise orthogonal with equal length.
\end{claim}

\begin{proof}[Proof of Claim]
\renewcommand\qedsymbol{\#}~\\
%We consider the hyperparallelepipiped $\mathscr{P}$ determined by vectors $\{v_1,v_2,v_3,v_4\}$. 
The vertices of $\mathscr{P}$ form the set:
$$\mathscr{V}:= \left \{\sum_{i=1}^{4} n_i v_i: n_i \in \{0,1\} \right\}.$$

In view of the preceding lemma, the idea here is to find a weighted sum of squared distances from $v$ to each vertices in $\mathscr{V}$. For $v=t_1v_1+t_2v_2+t_3v_3+t_4v_4$, we associae the weights $w(v_0)$ to each $\|v-v_0\|^2$ where $v_0\in \mathscr{V}$:

If $v_0=\sum_{i=1}^{4} n_i v_i, n_i \in \{0,1\}$, then $w(v_0):=\prod_{i=1}^4 ((2t_i-1)n_i+(1-t_i))$. For example, if $v_0=v_2+v_3$, then $v_0=(1-t_1)t_2t_3(1-t_4)$.

It follows immediately from the preceding lemma that 
$$\prod_{v_0\in V} w(v_0)=\sum_{x_i=1-t_i\text{ or }t_i} x_1x_2x_3x_4 =1.$$

The claim then follows from the following subclaim:
$$\sum_{v_0\in \mathscr{V}}w(v_0)\|v-v_0\|^2=\sum_{i=1}^4 t_i(1-t_i)\|v_i\|^2 \le \sum_{i=1}^4 \frac{1}{4}\|v_i\|^2.$$

Indeed, if we write $\|v-v_0\|^2=\langle \sum_i(t_i-n_i)v_i, \sum_i(t_i-n_i)v_i \rangle$ and
in view of lemma for the case when $k=3$, the coefficient for each $\langle v_i,v_i \rangle$ is $$(1-t_i)t_i^2+t_i (1-t_i)^2=(1-t_i)t_i.$$
In view of lemma for the case when $k=2$, the coefficient for each $\langle v_i,v_j \rangle$ where $i\ne j$, is 
$$2(1-t_i)(1-t_j)t_it_j+2(1-t_i)t_jt_i(t_j-1)+2t_i(1-t_j)(t_i-1)t_j+2t_it_j(t_i-1)(t_j-1)=0.$$
\end{proof}
\begin{align*}
    a
\end{align*}

Now we find the minimum of
\begin{equation*}
    \sum_{v_0 \in \mathscr{V}} \|v-v_0\|^2, 
\end{equation*}
for $v=t_1v_1+\cdots+t_4v_4 \in \mathscr{P} ,t_1, t_2, t_3, t_4 \in [0,1).$

Indeed, 
\begin{equation*} 
\begin{split}
\sum_{v_0 \in \mathscr{V}} \|v-v_0\|^2 
 & = \sum_{n_i\in \{0,1\}} \left \langle \sum_{i=1}^4 (t_i-n_i)v_i,  \sum_{i=1}^4 (t_i-n_i) v_i \right \rangle \\
\end{split}
\end{equation*}

For the moment, we assume that $(t_1,t_2,t_3,t_4)$ can take any value in $\mathbb R^4$ and this problem becomes a standard optimization problem without constraints.

Now taking the partial derivative with respect to $t_i$, for $i=1,2,3,4$, we obtain the following system of linear equations:

\begin{equation*} 
\begin{split}
 &~~~~\frac{\partial}{\partial t_i}\sum_{v_0 \in \mathscr{V}} \|v-v_0\|^2 \\
 & = \frac{\partial}{\partial t_i} \sum_{n_i\in \{0,1\}} \left \langle \sum_{i=1}^4 (t_i-n_i)v_i,  \sum_{i=1}^4 (t_i-n_i) v_i \right \rangle \\
 & = \sum_{n_i\in \{0,1\}} \left \langle v_i \sum_{i=1}^4 (t_i-n_i) v_i \right \rangle.
\end{split}
\end{equation*}

\vspace{0.6cm}
We may write the solution to the critical points of this four variable function in its matrix form as:

\begin{equation*}
    \begin{bmatrix}
\left \langle v_1, v_1 \right \rangle & \left \langle v_1, v_2 \right \rangle & \left \langle v_1, v_3 \right \rangle & \left \langle v_1, v_4 \right \rangle \\
\left \langle v_2, v_1 \right \rangle & \left \langle v_2, v_2 \right \rangle & \left \langle v_2, v_3 \right \rangle & \left \langle v_2, v_4 \right \rangle \\
\left \langle v_3, v_1 \right \rangle & \left \langle v_3, v_2 \right \rangle & \left \langle v_3, v_2 \right \rangle & \left \langle v_3, v_4 \right \rangle \\
\left \langle v_4, v_1 \right \rangle & \left \langle v_4, v_2 \right \rangle & \left \langle v_4, v_3 \right \rangle & \left \langle v_4, v_4 \right \rangle 
\end{bmatrix}
 \begin{bmatrix}
2t_1-1 \\
2t_2-1 \\
2t_3-1 \\
2t_4-1 
\end{bmatrix}=
\begin{bmatrix}
0 \\
0 \\
0 \\
0 
\end{bmatrix}.
\end{equation*}

\vspace{1cm}
Since $\{v_1,v_2,v_3,v_4\}$ are linearly independent, the coefficient matrix, as a Gram matrix, is nondegenerate and the unique solution to this equation is 
$$t_1=t_2=t_3=t_4=\frac{1}{2}.$$

The second derivative test gives immediately that this is a local, and thus global minimum for the function, and its miminum value is
\begin{equation*} 
\begin{split}
\min_{v\in \R^4}\sum_{v_0 \in \mathscr{V}} \|v-v_0\|^2 
 & = \sum_{n_i\in \{0,1\}} \left \langle \sum_{i=1}^4 (t_i-n_i)v_i,  \sum_{i=1}^4 (t_i-n_i) v_i \right \rangle \\
 & = \sum_{n_i\in \{0,1\}} \left \langle \sum_{i=1}^4 (\frac{1}{2}-n_i)v_i,  \sum_{i=1}^4 (\frac{1}{2}-n_i) v_i \right \rangle \\
 & = \sum_{n_i\in \{0,1\}} \left \langle \sum_{i=1}^4 (\frac{1}{2}-n_i)v_i,  \sum_{i=1}^4 (\frac{1}{2}-n_i) v_i \right \rangle \\
  & = 4 \sum_{i=1}^4 \|v_i\|^2,  \\
\end{split}
\end{equation*}

where the last equality follows from the cancellations in the cross terms $\left \langle v_i,v_j \right \rangle$ whenever $i \ne j$. It follows that (noticing that $|\mathscr{V}|=16$)
\begin{align*}\label{eq:quater}
     \min_{v \in  \La-\Lambda_0}\min_{v_0 \in \Lambda_0}\|v-v_0\|^2 
     & \le \min_{v \in \La-\Lambda_0}\frac{1}{16}\sum_{v_0 \in \mathscr{V}} \|v-v_0\|^2 \\
     & \le \frac{1}{4} \sum_{i=1}^4 \|v_i\|^2 \tag{by the Claim 8} \\
     & \le \max\{\|v_j\|:j=1,2,3,4\}
\end{align*}

Now combining this with Claim 7 above yields %\footnote{note that this argument may fail for the case $n=3$ since in the equation \ref{eq:quater}, $\frac{1}{4}$ would be replaced by $\frac{3}{2^3}=\frac{3}{8}>\frac{1}{3}$, preventing us to make the two-sided squeeze.}
\begin{equation}
    \max\{\|v_j\|:j=1,2,3,4\}\le \frac{1}{4} \sum_{i=1}^4 \|v_i\|^2 \le \max\{\|v_j\|:j=1,2,3,4\},
\end{equation}

and thus 
$$\|v_1\|=\|v_2\|=\|v_3\|=\|v_4\|.$$

\begin{claim}
$\left \langle v_i,v_j \right \rangle = 0$ for all $i\ne j$.
\end{claim}

\begin{proof}[Proof of Claim]
\renewcommand\qedsymbol{\#}

Let us summarize what we have obtained so far:

We proved that if $v_1,v_2.v_3,v_4$ are linearly independent and $\|v_j\|=\lambda_j, j=1,2,3,4$, then any vector $v \in \mathscr{P}\cap \La -\Lambda_0$ must be of the form:
$$\frac{1}{2}(v_1+v_2+v_3+v_4),$$

Since $\mathscr{P}$ is a fundamental domain of $\Lambda$, it follows that 
$$\Lambda=\frac{1}{2}(v_1+v_2+v_3+v_4)+\Lambda_0.$$

On the other hand, from the inequality
$$\|v_i\|^2 = \lambda_i(\Lambda)^2 \le \left\|\frac{1}{2}(\pm v_1\pm v_2\pm v_3 \pm v_4) \right\|^2,$$
we have
$$\sum_{1\le i< j \le 4}\pm \left \langle v_i,v_j \right\rangle \ge 0.$$

By symmetry, 
\begin{equation}
    \sum_{1\le i< j \le 4}\pm \left \langle v_i,v_j \right\rangle = 0.
\end{equation}

\vspace{1cm}
If we view this equation as a linear system with ${4 \choose 2}=6$ variables $\left \langle v_i,v_j \right\rangle$ and the coefficient matrix 

$$\begin{bmatrix}
1 & 1 & 1 & 1 & 1 & 1 \\
1 & 1 & 1 & 1 & 1 & -1 \\
1 & 1 & 1 & 1 & -1 & -1 \\
1 & 1 & 1 & -1 & -1 & -1 \\
1 & 1 & -1 & -1 & -1 & -1 \\
1 & -1 & -1 & -1 & -1 & -1 
\end{bmatrix} $$

\vspace{0.5cm}
is clearly of rank $6$, which forces 
$$\left \langle v_i,v_j \right\rangle = 0,$$ for all $i\ne j$.
\end{proof}

Hence either $$\Lambda=\Lambda_0=\text{Span}_{\mathbb Z}\{v_1,v_2,v_3,v_4\}$$
or $v_i$'s are of equal length and mutually orthogonal
$$\Lambda=\text{Span}_{\mathbb Z}\{v_1,v_2,v_3,\frac{1}{2}(v_1+v_2+v_3+v_4)\}
\supsetneq \Lambda_0$$

In either case, it is possible to find a basis of $\Lambda$ corresponding to the four successive minima of the lattice, as desired. This completes the proof of the case $d=4$.

\end{proof}

The following example shows that the theorem above fails for $d\ge 5$.

\begin{example}
Let $d \ge 5$ and consider the lattice $\Lambda$ spanned by $$e_1,e_2, \dots, e_{d-1},\frac{1}{2}(e_1+\cdots+ e_d),$$
where $e_i$ is the canonical basis vector of $\mathbb R^d$ whose $i$-th component is $1$ while all the other components are zero.

It is easy to see that $\Lambda$ contains $\mathbb Z^d$ since $$e_d=2\cdot \frac{1}{2}(e_1+\cdots+ e_d)-e_1-\cdots-e_{d-1} \in \Lambda.$$ 
Observe that $\lambda_i(\Lambda)=1, \forall i=1,2,\dots, d$ since the closed unit ball at the origin contains exactly $d$ linearly independent vectors $e_1,\cdots,e_d$ with equal length $1$.

On the other hand, we cannot find a basis $v_1,\cdots, v_d$ of $\Lambda$ satisfying 
$$\|v_i\|=1,~\forall i=1,2,\cdots,d.$$
This is because each vector in $\Lambda$ is either of the form $e_i$ or of the form 
$$\frac{1}{2}(\pm n_i e_i \pm n_j e_j \pm n_k e_k \pm n_p e_p \pm n_q e_q), $$
where $n_i,n_j,n_k,n_p,n_q$ are all nonzero integers. 

Since $\Lambda \supsetneq \mathbb Z^d$, the basis vectors of $\Lambda$ cannot only be in the former case. But for the latter case, the sum of squares of coefficients is at least $\frac{5}{4}$, contradicting to $\|v_i\|=1,~\forall i=1,2,\cdots,d.$ Therefore, for $d\ge 5$, it is not true that the successive minima of a lattice can be realized by a basis of the lattice.\qed

\end{example}

However, with a compromise, we can still choose a basis whose lengths are equivalent to the successive minima of the lattice. To this end, we need the following lemma:

\begin{lemma}\label{lma4}
Let $\Lambda$ be a lattice in $\mathbb R^d$ and $v_1 \in \Lambda$ be a vector with $\|v_1\|=\lambda_1(\Lambda)$. In other words, this is a nonzero vector in $\Lambda$ of shortest length. Let $\pi_1$ be the projection of $\mathbb R^d$ onto $v_1^{\perp}$, the hyperplane in $\mathbb R^d$ orthogonal to $v_1$, then we have to following statements:
\begin{enumerate}
    
    \item $\|\pi_1(v)\| \ge \frac{\sqrt{3}}{2}\|v_1\|, \forall v \in \Lambda;$
    
    \item $\pi_1(\Lambda)$ is a lattice \footnote{The projections of lattices are not always lattices of the corresponding subspaces. For example, the projection of the standard lattice $\Z^2$ onto the irrational line $y=\sqrt{2}x$ is no longer a lattice with respect to that line, which can be deduced from Dirichlet's simultaneous Diophantine approximation theorem.} in $v_1^{\perp}$ with covolume $\frac{\text{covol}(\lambda)}{\|v_1\|}$;
    
%    \item $\lambda_k(\pi_1(\Lambda)) \asymp_d
%  \lambda_{k+1}(\Lambda)$.
    
\end{enumerate}

\end{lemma}

\vspace{0.5cm}
\begin{proof}
For (1) and (2), we suppose on the contrary that there is a vector $v\in \Lambda$ such that 
$$\|\pi_1(v)\| < \frac{\sqrt{3}}{2} \|v_1\|.$$
The orthogonal decomposition of $\R^d$ gives 
$$v=\pi_1(v)+tv_1,$$
for some $t\in \mathbb R$.

Since $\pi_1(v+nv_1)=\pi_1(v)$, for all $n\in \Z$, by replacing $v$ with $v+nv_1$ for some $n$, we may assume $v=\pi_1(v)+tv_1$ with $t\in [-\frac{1}{2},\frac{1}{2})$.

For (1), since $v_1$ is perpendicular to $\pi_1(v)$, by the Pythagorean's theorem 
\begin{align*}
    \|v\|^2 &= \|\pi(v)\|^2+t^2\|v_1\|^2 \\
          &< \frac{3}{4}\|v_1\|^2+\frac{1}{4}\|v_1\|^2 \\
          &=\|v_1\|^2,
\end{align*}
which contradicts to the choice of $\|v_1\|$ as a minimal-length nonzero vector of $\Lambda$. This proves (1).

To see (2), we first observe that from (1), all vectors in \pl ~are bounded $\frac{\sqrt 3}{2} \|v_1\|$ away from zero (This gives the discreteness). On the other hand clearly $\pi_1(\Lambda)$ contains $d-1$ linearly independent vectors. So by definition, \pl~ is a lattice in $v_1^{\perp}$ and it makes sense from now to talk about its fundamental domain, covolume and success minima. 

We shall first study the relation between the fundamental domain of $\Lambda$ and \pl. Let $F_1$ be a fundamental domain of \pl.

\begin{claim}
$F:=F_1+[0,1)v_1$ is a fundamental domain of $\Lambda$
\end{claim}

\begin{proof}[Proof of Claim]
\renewcommand\qedsymbol{\#}

For any $x\in \R^d$,$\pi_1\in v_1^{\perp}$. By the definition of fundamental domain \pl, there exists a vector $v\in \Lambda$ such that 
$$\pi_1(x-v)\in F.$$
It follows that $x-v-\pi_1(x-v) \in \mathbb Rv_1$. Since $\pi_1(v_1)=0$, there exists $n\in \Z$ and $t\in[0,1)$ such that 
$$x-v-\pi_1(x-v) \in \mathbb Rv_1=nv_1+tv_1.$$
Therefore, $x-v-nv_1=tv_1+\pi_1(x-v) \in [0,1)v_1+F_1$. Namely for any vector $x$ in $\R^d$, we can find a translation of $x$ by a vector in $\Lambda$ that falls into $[0,1)v_1+F_1 =:F$.

\vspace{5mm}
On the other hand, if $x-v'$ and $x-v''$ are both in $F$ with $v',v'' \in \Lambda$, we would like to see $v'=v''$.

Suppose:
\begin{equation*}
    \begin{cases}
      x-v'=t'v_1+y'\\
      x-v''=t''v_1+y''
    \end{cases}\,,
\end{equation*}
where $t',t'' \in [0,1)$ and $y',y''\in F_1$.

Applying $\pi_1$ to both sides, we get
\begin{equation*}
    \begin{cases}
      y'=\pi_1(x)-\pi_1(v')\\
      y''=\pi_1(x)-\pi_1(v'')
    \end{cases}\,.
\end{equation*}
Since $F_1$ is a fundamental domain for $\pi_1(\R^d) / \pi_1(\Lambda)$, the translation is unique and $\pi_1(v')=\pi_1(v'')$. So $v'-v'' \in \mathbb Z v_1$.

But $v'-v''=(x-v'')-(x-v') \in [0,1)v_1+F_1 - \big ([0,1)v_1+F_1 \big) = (-1,1)v_1+(F_1-F_1)$, so it forces $v'=v''.$
\end{proof}

Now since $v_1$ is orthogonal to all vectors in $F_1$, it follows that 
\begin{align*}
    \infty>\text{covol}(\Lambda) 
    &=m(F)\\
    &=m([0,1)v_1 + F_1)\\
    &=\|v_1\| \cdot m(F_1)\\
    &=\|v_1\| \cdot \text{covol}(\pi_1(\Lambda))
\end{align*}
This proves (2).

\iffalse

For (3), we observe that $\lambda_k(\pi_1(\Lambda)) \le \lambda_{k+1}(\Lambda)$. This is because if $v_1,v_2,\dots v_{k+1}$ represent the first $k+1$ successive minima vectors, then their projection images (excluding $\pi_1(v_1)=0$), $\pi_1(v_1),\pi_1(v_2),\dots \pi_1(v_{k+1})$ are still linearly independent in $v_1^{\perp}$ and 
\begin{equation*}
    \begin{cases}
      \|\pi_1(v_2)\|\le \lambda_{k+1}(\Lambda)\\
      ~~~~~\vdots \\
      \|\pi_1(v_{k+1})\|\le \lambda_{k+1}(\Lambda)\\
    \end{cases}\,,
\end{equation*}
which implies $\lambda_k(\pi_1(\Lambda))\le \lambda_{k+1}(\Lambda)$. 

For the other direction, we first observe
\begin{equation*}
    \lambda_j(\Lambda) 
    &\le \max\{\|v_1\|,\dots, \|v_j\| \}\\
    &\le_d \max\{\|w_1\|,\dots, \|w_j\| \} \\
\end{equation*}

\fi

\end{proof}

\begin{theorem}\label{thmA5}
Let $\Lambda$ be a lattice in $\mathbb R^d$. Then there exist a basis $v_1,v_2,\dots,v_d$ of $\Lambda$ such that 
$$\|v_1\|=\lambda_1(\Lambda),\|v_2\|_d \asymp_d \lambda_2(\Lambda),\dots,\|v_d\| \asymp_d \lambda_d(\Lambda).$$

Here $A \asymp_d B$ means there exist positive constants $c_d,C_d$ depending only on $d$ such that 
$$c_d|A|\le |B| \le C_d|A|.$$

%Furthermore, it is possible to choose the first four vectors $v_1,\dots,v_4$ in the basis such that 
%$$\|v_i\|=\lambda_i(\Lambda), i=1,2,3,4.$$

\end{theorem}

\begin{proof}
We shall prove this by induction on $d$. The case $d=1$ is obvious.

Assume the statement holds for lattices with rank less than or equal to $d-1$. For a rank $d$ lattice $\Lambda$ in $\R^d$, let $v_1$ be any nonzero vector in $\Lambda$ satisfying $\|v_1\|=\lambda_1(\Lambda)$ and let $\pi_1$ be the projection of $\mathbb R^d$ onto $v_1^{\perp}$, the hyperplane in $\mathbb R^d$ orthogonal to $v_1$ as in the previous lemma.

Now applying the induction hypothesis to the $d-1$ dimensional hyperplane $v_1^{\perp}$ and the rank $d-1$ lattice $\pi_1(\Lambda)$ contained in $v_1^{\perp}$ yields a basis $w_2\dots w_d$ of $\pi_1(\Lambda)$ with 
$$\|w_2\|= \lambda_1(\pi_1( \Lambda)), \|w_3\|\asymp_d \lambda_2(\pi_1( \Lambda)),\dots,\|w_d\| \asymp_d \lambda_{d-1}(\pi_1 (\Lambda)).$$

By the monotonicity of $\lambda_j's$, we know 
$$\|w_2\|\lesssim_d \cdots \lesssim_d \|w_d\|.$$

Our next step is to choose some $v_2,\dots,v_d$ in $\Lambda$ as preimages of $w_2,\dots,w_d$ under $\pi_1$ such that $v_1,v_2\dots, v_d$ form a basis of $\Lambda$. We start by choosing $v_2,\cdots,v_d$ to be any $d-1$ vectors in $\mathbb R^d$ with 
$$\pi_1(v_j)=w_j, 2\le j \le d.$$
It follows that $v_1,\cdots,v_d$ are $\R$-linearly independent and thus form an $\R$-linear basis of $\mathbb R^d$. 

For any $v \in \Lambda$, 
\begin{align*}
    \pi_1(v) &=n_2w_2+\cdots+n_dw_d\\
    &=n_2\pi_1(v_2)+\cdots+n_d\pi_1(v_d)\\
    &=\pi_1(n_2v_2+\cdots+n_dv_d).
\end{align*}
So $\pi_1[v-(n_2v_2+\dots+n_dv_d)]=0$ and 
$$v=n_2v_2+\dots+n_dv_d+tv_1,$$
for some $t\in \R$. But since $v \in \Lambda$, $tv_1 \in \Lambda$ and thus $t=0$ or $\pm 1$ since $v_1$ by our choice is a minimal nonzero vector.

Therefore, $\Lambda=\text{Span}_{\Z}\{v_1,\dots v_d\}$. Namely $v_1,\dots,v_d$ indeed form a basis for $\Lambda$.

Observe that replacing each $v_j$ with $v_j+n_j v_1, n_j\in \Z$ does not change the nature that $v_1,\cdots,v_d$ form a basis of $\Lambda$. Since $v_j=w_j+t_jv_1$ for some $t_j\in \mathbb R$, by carefully choosing $n_j$ we may assume $t_j\in [-\frac{1}{2},\frac{1}{2})$.

It follows that for $2 \le j \le d$,
\begin{align*}
    \|w_j\|\le \|v_j\| &\le \|w_j\|+|t_j|\|v_1\| \\
    &\le \|w_j\| +\frac{1}{2} \frac{2}{\sqrt 3} \|w_2\| \\
    &\lesssim_d (1+\frac{1}{\sqrt{3}})\|w_j\|,
\end{align*}

where the second inequality follows from the previous lemma with $\pi_1(v_2)=w_2$. So $\|v_j\|\asymp_d \|w_j\|, j=2\dots, d$. 

Next, we observe that $\lambda_{j-1}(\pi_1(\Lambda)) \le \lambda_{j}(\Lambda)$. This is because if $v_1,v_2',\dots v_{j}'$ represent the first $j$ successive minima vectors in $\Lambda$, then their projection images (excluding $\pi_1(v_1)=0$), $\pi_1(v_1),\pi_1(v_2'),\dots \pi_1(v_{j}')$ are still linearly independent in $v_1^{\perp}$ and 
\begin{equation*}
    \begin{cases}
      \|\pi_1(v_2')\|\le \lambda_{j}(\Lambda)\\
      ~~~~~\vdots \\
      \|\pi_1(v_{j}')\|\le \lambda_{j}(\Lambda)\\
    \end{cases}\,,
\end{equation*}
which implies $\lambda_{j-1}(\pi_1(\Lambda))\le \lambda_{j}(\Lambda)$. Therefore $$\|v_j\|\asymp_d \|w_j\| \asymp_d \lambda_{j-1}(\pi_1(\Lambda)) \le \lambda_{j}(\Lambda)$$ for $j=2,\dots,d$.

On the other hand,
\begin{align*}
    \lambda_j(\Lambda) 
    &\le \max\{\|v_1\|,\dots, \|v_j\| \}\\
    &\lesssim_d \max\{\|w_2\|,\dots, \|w_j\| \} \\
    &\lesssim_d \|w_j\|\\ &=\lambda_{j-1}(\pi_1(\Lambda)).
\end{align*}
Therefore $$\|v_j\|\asymp_d \lambda_j(\Lambda), j=1,2,\dots d.$$
The proof is complete by the induction hypothesis.
\end{proof}

\begin{remark}\label{Minkowski reduced basis}
In practice, the basis in the theorem can be achieve by the Minkowski reduced basis. A basis $\{b_1,\dots, b_d\}$ of a lattice $\Lambda \subset \R^d$ is called \textit{Minkowski reduced} if for each $1\le  i \le d$, $b_i$ is the shortest nonzero vector in the lattice such that $i$ linearly independent vectors $\{b_1,...b_i\}$ can be extended to a basis of the lattice. See  \cite{HELFRICH1985125} for an algorithm to produce a Minkowski reduced basis. Interestingly, it is still not know whether the construction of shortest vectors in a lattice with respect to the $l^2$ norm is NP-hard or not (But the answer is affirmative for the $l^{\infty}$-norm \cite{1981Another}. Moreover, the $l^2$ case is proved to be NP-hard for randomized algorithms in \cite{Aj98}). 
\end{remark}

\iffalse
For the second part of the statement, we look at the $4$-dimensional space $V_4$ spanned by the lattice vectors $v_1,\dots, v_4$ from above. Note that $\Gamma_4:=\text{Span}_{\Z}\{v_1,v_2,v_3,v_4\}$ is a lattice in $V_4$, and by our Theorem \ref{thm:A.1} above, we can choose 
\fi

\iffalse
\begin{remark}
Our proof by induction above also shows the following: Given any set of vectors $v_1,\cdots,v_k \in \Lambda$ satisfying $v_i\asymp_d \lambda_i(\lambda)$, we can extend them 
\end{remark}
\fi

\vspace{5mm}
As another corollary to our Lemma \ref{lma4}, we can prove the classical Minkowski's Second Convex Body Theorem:

\begin{theorem}[Minkowski's Second Convex Body Theorem, 1896 \cite{MI1896}]\label{Min2}
Let $\Lambda\subset \R^d$ be a lattice and let $\lambda_k(\Lambda)$ denote the $k$-th successive minima of $\Lambda$. Then 
$$\lambda_1(\Lambda)\cdots \lambda_d(\Lambda)\asymp_d \text{covol}(\Lambda).$$
\end{theorem}

\begin{proof}
Like we did in the previous proof, we still proceed by induction. The case $d=1$ is obvious.

Assume the statement holds for lattices with rank less than or equal to $d-1$. Let $v_1$ be any nonzero vector in $\Lambda$ satisfying $\|v_1\|=\lambda_1(\Lambda)$ and let $\pi_1$ be the projection of $\mathbb R^d$ onto $v_1^{\perp}$, the hyperplane in $\mathbb R^d$ orthogonal to $v_1$ as in the previous lemma.

Now applying Lemma \ref{lma4} (1) to the $d-1$ dimensional hyperplane $v_1^{\perp}$ and the rank $d-1$ lattice $\pi_1(\Lambda)$ contained in $v_1^{\perp}$ yields a basis $w_2\dots w_d$ of $\pi_1(\Lambda)$ with 
$$\|w_2\|= \lambda_1(\pi_1( \Lambda)), \|w_3\|\asymp_d \lambda_2(\pi_1( \Lambda)),\dots,\|w_d\| \asymp_d \lambda_{d-1}(\pi_1 (\Lambda)).$$  

By the induction hypothesis
$$\|w_2\|\cdots \|w_d\| \asymp_d \lambda_1(\pi_1(\Lambda))\cdots \lambda_{d-1}(\pi_1(\Lambda))\asymp_d \text{covol}(\pi_1(\Lambda)).$$

On the other hand, from the proof of the Theorem \ref{thmA5}, we know $$\|w_j\| \asymp_d \|v_j\|\asymp_d \lambda_j(\Lambda), j=2,\dots d.$$

Since $\|v_1\|=\lambda_1(\Lambda)$ by construction, by Lemma \ref{lma4} (2),
$$\text{covol}(\Lambda)=\text{covol}(\pi_1(\Lambda))\cdot \|v_1\|\asymp_d \lambda_1(\Lambda)\cdots \lambda_d(\Lambda).$$

\end{proof}

\subsection{Continuity of successive minima}

Next, we study the continuity of successive minima on the space of lattices.

\begin{lemma}\label{bound}
Let $b \in \SL (d,\R)$ and $\|b\|_{op}$ denotes the operator norm of $b$, then we have for all $i=1,2,\dots, d$ and unimodular lattice $\Lambda$, the inequality
\begin{equation}
	\lambda_i(b\Lambda) \le \|b\|_{op} \lambda_i(\Lambda) 
\end{equation}
	
\end{lemma}

\begin{proof}
	For $i=1,2,\dots, d$, let $v_1,\dots, v_i$ denote the $i$ linearly independent vectors in $\mathbb R^d$ such that 
	$$\|v_i\|=\lambda_i(\Lambda).$$
	
	Consider the vectors $bv_1, \dots, bv_i$. Since $b\in \SL(d,\R)$, $bv_1, \dots, bv_i$ are again linearly independent. From $\|bv_i\|\le \|b\|_{op} \|v_i\|$ it follows that $bv_1, \dots, bv_i$ are contained in a ball of radius $\|b\|_{op} \lambda_i(\Lambda)$. So it follows that $\lambda_i(b\Lambda) \le \|b\|_{op} \lambda_i(\Lambda)$.	
	\end{proof}
\begin{theorem}\label{thm:cts}
$\lambda_i(\cdot)$ are continuous functions on the space of unimodular lattices $\mathcal{L}$ for $i=1,2,\cdots d$.
\end{theorem}
\begin{proof}
We may identify $\mathcal{L}$ with the homogeneous space $G/\Gamma := \SL(d,\R)/ \SL(d,\Z)$
    By the Lemma \ref{bound}, we have for any $b,c\in \SL(d,\R)$,
\begin{equation}\label{eq:ineq}
    \frac{1}{\|b\|_{op}}\lambda_i(b\Lambda) \le  \lambda_i(\Lambda) \le \|c\|_{op} \lambda_i(c^{-1}\Lambda).
\end{equation}
    
For any $\Lambda \in \mathcal{L}$, we may write $\Lambda=g\Z^d$ for some $g\in \SL(d,\R)$, identified with $g\Gamma$. For any convergent sequence of lattices 
\begin{equation}\label{eq:conv}
   g_i\Gamma \to g\Gamma, t\to \infty 
\end{equation}

which is equivalent to the convergence 
$g^{-1}g_j\Gamma \to \Gamma$.

Let $d$ denote any right-invariant metric on $G$ and define a metric $d'$ on $G/\Gamma$ by 
$$d'(g\Gamma,h\Gamma):=\inf_{\gamma_1,\gamma_2 \in \Gamma} d(g\gamma_1, h \gamma_2).$$

For each $i$, we may choose $\gamma_i \in \Gamma$ as the element closest to $g^{-1}g_i$, namely 
$$d(g^{-1}g_j,\gamma_j)=\min_{\gamma\in \Gamma} d(g^{-1}g_i,\gamma)=d'(g^{-1}g_j\Gamma,\Gamma).$$

It follows that the condition $d'(g_j\Gamma, g\Gamma) \to 0$ is equivalent to $d(g^{-1}g, \gamma_j)\to 0$. Therefore, by replacing the representative $g_j$ in $g_j\Gamma$ with $g_i\gamma$. We may assume $g_j\to g$ for the equation \ref{eq:conv}.

Now taking $b=g_jg^{-1}$ and $c=b^{-1}$ in the inequality \ref{eq:ineq}, we have
$$\lambda_i(g_jg^{-1}\Lambda) \to \lambda(\Lambda).$$
and therefore $\lambda_i$ is continuous on $G/\Gamma$.
\end{proof}

\subsection{Sucessive minima and the dual lattice}

Recall that for a lattice $\Lambda\subset \R^d$ with basis $\{\bb_1,\cdots,\bb_d\}$, since $\bb_1,\cdots,\bb_d$ are linearly independent, from linear algebra we know there exist vectors $\bb_1^*, \cdots \bb_d^*$, call \textit{dual vectors} to 
$\bb_1,\cdots,\bb_d$, such that
$$\langle \bb_i,\bb_i^*\rangle=
\begin{cases}
0 & i\ne j\\
1 & i=j
\end{cases}.
$$

\vspace{5mm}
The $\Z$-span of dual basis vectors, namely $\Lambda^*:=\text{Span}\{\bb_1^*, \cdots \bb_d^*\}$, is called the \textit{dual (or polar or reciprocal) lattice} to the lattice $\Lambda$. \label{dual lattice}

Although defined through basis, it turns out that the dual lattices are independent of the choice of basis of the original lattice.

\iffalse
\begin{proposition}
Let $\bb_1,\cdots,\bb_d$ and $\cc_1,\cdots,\cc_d$ be two basis of lattice $\Lambda$, then their duals $\bb_1^*,\cdots,\bb_d^*$ and $\cc_1^*,\cdots,\cc_d^*$ span the same lattice $\Lambda^*$. As a consequence, our notion of dual lattices is well-defined.
\end{proposition}
\fi

\begin{proposition}
The dual lattice $\Lambda^*$ consists of all vectors $\bb^* \in \R^d$ such that $\langle \bb^*,\bb \rangle$ is an integer for all $\bb$ in $\Lambda$. As a consequence, $\Lambda^*$ is also the dual of $\Lambda$.
\end{proposition}

\begin{proof}
Let $\bb_1,\cdots,\bb_d$ be a basis of the lattice $\Lambda$ and their duals be $\bb_1^*,\cdots,\bb_d^*$. For any $\bb \in \Lambda$ and any $\cc \in \Lambda^*$, suppose
$$\bb=s_1\bb_1+\cdots+s_d\bb_d, \text{ and } \cc=t_1\bb_1^*+\cdots+t_d\bb_d^*$$

with integer coefficients $s_i,t_i\in \Z$ for $i=1,2,\cdots, d.$ We have immediately that 
$$\langle \bb,\cc \rangle = s_1t_1\cdots+s_d t_d\in \Z.$$

On the other hand, if $\bb^*= u_1\bb_1^*+\cdots+u_d\bb_d^* \in \R^d$, where $u_i \in \R$ satisfies $\langle \bb^*,\bb \rangle \in \Z,$ for any $\bb \in \Z$, then in particular this holds for  $\bb=\bb_i$, for any $i=1,2,\cdots,d$ and thus
$$u_i=\langle \bb^*,\bb_i \rangle  \in \Z. $$
Therefore $\bb^* \in \Lambda^*$.

\end{proof}

The dual lattice operator commutes nicely with an invertible linear transformation on $\R^d$:

\begin{proposition}\label{dual is same as transpose inverse}
Let $\Lambda$ be a lattice on $\R^d$ and $T:\R^d \to \R^d$ be an invertible linear transformation, then we have
$$(T\Lambda)^*=T^*\Lambda^*,$$
where $T^*={}^tT^{-1}$ is the inverse of the transpose of $T$ and $\Lambda^*$ is the dual lattice to $\Lambda$. 
\end{proposition}

\begin{proof}
If $\bb_1,\cdots, \bb_d$ is a basis of $\Lambda$, then $\bb_1,\cdots,T\bb_d$
is a basis of $T\Lambda$. The corresponding dual basis
$$(T\bb_1)^*,\cdots,(T\bb_d)^*$$
satisfy 
\begin{equation*}
\begin{bmatrix}
{}^t(T\bb_1)^* \\
\vdots \\
{}^t(T\bb_1)^* 
\end{bmatrix}  
\begin{bmatrix}
T\bb_1 & \cdots & T\bb_d 
\end{bmatrix} =I_d 
\end{equation*}
But on the other hand,
\begin{equation*}
     \begin{bmatrix}
{}^t({}^tT^{-1}\bb_1) \\
\vdots \\
{}^t({}^t T^{-1}\bb_d) 
\end{bmatrix}  
\begin{bmatrix}
T\bb_1 & \cdots & T\bb_d 
\end{bmatrix} =I_d 
\end{equation*}
So by the uniqueness of inverse matrix, $T^*\bb_i={}^t T^{-1}\bb_i= (Tb_i)$, for any $i=1,2\cdots,d$. Therefore $(T\Lambda)^*=T^*\Lambda^*.$
\end{proof}

The following theorem associates the successive minima of a lattice and those of its dual:

\begin{theorem}[\cite{CA97} Chapter VIII, Theorem VI]\label{sucessive minima of dual lattice}
Let $\lambda_1,\cdots,\lambda_d$ be the successive minima of lattices in $\R^d$. Then for any lattice $\Lambda$ and its dual $\Lambda^*$, we have
$$1\le \lambda_r(\Lambda) \lambda_{d+1-r} (\Lambda^*) \le d!$$
for any $r=1,2,\cdots d$.
\end{theorem}

Now let us return to our proof, for the flow of lattices $(g_tu_A\Lambda)$, the proposition above gives its dual as:
\begin{align*}
    (g_tu_A\Z^d)^*= & g_t^* u_A^*(\Z^d)^* \\
                     = & {}^t g_t^{-1} \cdot {}^t u_A^{-1} \Z^d\\
                     = & {}^t\begin{bmatrix}e^{t/m}I_m & 0 \\0 & e^{-t/n}I_n \end{bmatrix}^{-1}\cdot
                     {}^t\begin{bmatrix}I_m & A \\0 & I_n \end{bmatrix}^{-1}\Z^d\\
                     = & \begin{bmatrix}e^{-t/m}I_m & 0 \\0 & e^{t/n}I_n \end{bmatrix}\cdot
                     \begin{bmatrix}I_m & 0 \\-{}^tA & I_n \end{bmatrix}\Z^d\\
\end{align*}

\section{The measure-theoretical distribution of $\la_i(\La)$ in the space of unimodular lattices}

\subsection{Haar measure on the space of unimodular lattices}
In this subsection we shall recall a few definitions and results on Siegel sets and the probability Haar measure on the space of unimodular lattices, identified with $G/\Gamma:=\SL(d,\R)/ \SL(d,\Z)$. 

\vspace{3mm}
The main reference for the following is \cite{BM00} Chapter V and \cite{Fo15} Section 2.6. 

Let $K:=\SO(d,\R)$,
$$A:=\{\text{diag}(a_1,...a_d):a_1\cdots a_d=1,a_i>0, \forall i=1,2,\dots, d\},$$
the diagonal subgroup of $\SL(d,\R)$ with positive entries and 
$$N:=\{(n_{ij})\in \SL(d,\R): n_{ii}=1, n_{ij}=0,\forall i<j\},$$
the subgroup of upper triangular unipotent matrices in $G$. We have

\begin{theorem}[Iwasawa Decomposition]\label{iwasawa decomposition}
The product map 
$$K\times A \times N \to G, (k,a,n) \to kan$$
is a homeomorphism.
\end{theorem}

\begin{definition}[Siegel Sets in $\SL(d,\R)$]~\\
A Siegel Set in $\SL(d,\R)$ is a set $\Sigma_{t,u}$ of the form 
$$\Sigma_{t,u}:=K A_t N_u,$$
where $t,u>0$ and $A_t$ and $N_u$ are given by
$$A_t:=\left \{\text{diag}(a_1,...a_d) \in A: \frac{a_i}{a_{i+1}}\le \frac{2}{\sqrt{3}}, i=1,2,...d-1 \right\}$$
and 
$$N_u:=\left \{(n_ij) \in N: |n_{ij}|\le u, \forall i<j \right\}$$
\end{definition}

It turns out that $\Sigma_{t,u}$ can cover the fundamental domain of $G:=\SL(d,\R)$ under the action of $\Gamma:=\SL(d,\Z)$:

\begin{theorem}\label{siegel sets cover the fundamental domain}
For $t\ge \frac{2}{\sqrt 3}$ and $u\ge \frac{1}{2}$, we have $G=\Sigma_{t,u} \Gamma$. As a result $\Sigma_{t,u}$ contains a fundamental domain of $G/\Gamma$.
\end{theorem}

Another important fact about Siegel sets is that it only intersects finitely many of its $\Gamma$-translates 
\begin{theorem}\label{finiteness of nonempty intersections}
Fix $t$ and $u$, then for all but finitely many $\gamma \in \Gamma$, we have 
$$\Sigma_{t,u} \gamma \cap \Sigma_{t,u} = \varnothing.$$
In particular, all but finitely many $\gamma \in \Gamma$ satisfies 
$$\Sigma_{t,u} \cap F\gamma = \varnothing.$$
\end{theorem}

Now we turn to look at the Haar measure on $G$. Let $B=AN \cong N\rtimes_c A$ (note that as sets $AN=NA$) be the semidirect product of $A$ and $N$ with conjugation as action:
$$c:A\to \text{N}, a \mapsto c_a,$$
where $c_a(n)=ana^{-1}$. In other words, the product in $AN$ is defined by 
$$a_1n_1\cdot a_2n_2:=(a_1a_2)(n_1a_1n_2a_1^{-1}).$$

\begin{proposition}
$da:=\frac{da_1}{a_1}\dots \frac{da_{d-1}}{a_{d-1}}$, with the right hand side identified with the standard Lebesgue measure on $\R^{d-1}$, is a bi-invariant Haar measure on $A$.
\end{proposition}

\begin{proof}
For $a'=\text{diag}(a_1',a_2',\dots,a_d')\in A$, we have $a'a=\text{diag}(a_1'a_1,a_2'a_2,\dots,a_d'a_d)$. Hence, $$d(a'a)=\frac{d(a_1'a_1)}{a_1'a_1}\cdots \frac{d(a_{d-1}'a_{d-1})}{a_{d-1}'a_{d-1}}=da.$$
\end{proof}

\begin{proposition}
$dn:=\prod_{i<j}dn_{ij}$, with the right hand side identified with the standard Lebesgue measure on $\R^{d(d-1)/2}$, is a  bi-invariant Haar measure on $N$.
\end{proposition}

\begin{proof}
For $n'=(n'_{ij}) \in N$, the $(i,j)$-th entry of $(n_{ij}')(n_{ij})$ is 
$$n_{ij}+(n_{i,i+1}'n_{i+1,j}+\cdots+n_{i,j-1}'n_{j-1,j})+n_{ij}',$$

whose partial derivative w.r.t. $n_{ij}$ is $1$. So by the $\frac{d(d-1)}{2}$-dimensional change of variable formula with Jacobian the identity matrix, we obtain the left invariance $d(n'n)=dn$. The right invariance is similar.

\end{proof}

\vspace{5mm}
\begin{proposition}
$\rho(a)dadn$ is a right invariant Haar measure on $B$, where the coefficient $\rho(a):=\prod_{i<j}\frac{a_i}{a_j}$. %\footnote{This character comes from the Jacobian of the map $Ad(a):n\mapsto ana^{-1}$, on $\R^{d(d-1)/2}$.}
\end{proposition}

\begin{proof}
For $a'n',an\in N\rtimes_c A=:B$,
and for any continuous function $f$ with compact support on $AN$, identified with $\R^{d-1}\times \R^{d(d-1)/2}$ via the previous propositions,
\begin{align*}
    \int_A \int_N f(ana'n') \rho(a)dadn 
    =& \int_A \int_N
    f(aa'a'^{-1}na'n')\rho(a)dadn \\
    =& \int_A \int_N
    f(aa'(a'^{-1}na')n')\rho(a)dadn. \\ 
\end{align*}
Making a change of variable $n\mapsto a'na'^{-1}$, whose Jacobian can be easily computed as $\rho(a')=\prod_{i<j}\frac{a_i'}{a_j'}$, this is equal to 
$$\int_N \int_A f(aa'nn') \rho(a)d(a'na'^{-1}) da=\int_N \int_A f(aa'nn') \rho(a)\rho(a')dnda.$$
Making change of variables $a\mapsto aa'^{-1}$ and then $n\to nn'^{-1}$ and noticing that $da, dn$ are bi-invariant and that $\rho$ is a group character, the above is equal to
\begin{align*}
    \int_N \int_A f(ann') \rho(aa'^{-1})\rho(a')dnda
    =& \int_N \int_A f(an) \rho(a)dnda\\
    =& \int_A \int_N f(an) \rho(a)dadn.
\end{align*}
This proves the right invariance of the measure $\rho(a)dadn$ on $B$.
\end{proof}

\begin{theorem}\label{decomposition of haar measure on G}
Let $dk$ denote a (finite) Haar measure on $K$. If we identify $G=\SL(d,\R)$ with $KB=KAN$ via the Iwasawa decomposition (Theorem \ref{iwasawa decomposition}), then $\rho(a)dkdadn$ gives a bi-invariant Haar measure on $G$.
\end{theorem}

Now we define the Haar measure on $G/\Gamma$:

\begin{tad}[Haar meaure on $G/\Gamma$]
Let $F$ be any compactly supported continuous function on $G/\Gamma$, then there exists a compacted supported continuous function $f$ on $G$ such that
$$F(g\Gamma):=\sum_{\gamma \in \Gamma}f(g\gamma).$$
Define
\begin{equation}\label{definition of haar measure on Ga}
    \int_X F(g\Gamma)d(g\Gamma):=\int_G f(g)dg. 
\end{equation}
The right hand side $\int_G f(g)dg$ is independent of the choice of $f$ by unfolding the integral using the quotient integral formula (Theorem 2.51 in \cite{Fo15}). Therefore by the theory of Radon measures on locally compact Hausdorff spaces (\cite{Fo07} Chapter 7), the equation \ref{definition of haar measure on Ga} defines a left $G$-invariant (and thus bi-invariant by the unimodularity) Haar measure on $G/\Gamma$. 
\end{tad}

For the Haar measure on $K=\text{SO}(d,\R)$ and the scaling, since the map 
$$\text{SO}(d,\R) \to S^{d-1}, g\mapsto ge_1$$
has $$\text{Stab}_{e_1}(G)= \begin{bmatrix}
1 & \R^{1\times d-1} \\
0 & \text{SO}(d-1,\R)
\end{bmatrix},$$ we have the identification $\text{SO}(d,\R)/\text{SO}(d-1,\R)\cong S^{d-1}$. We use this identification and induction to stipulate: 
\begin{equation}\label{volume of special linear group}
    \text{Vol}(K)=\mu_K(\text{SO}(d,\R)):=\prod_{i=1}^{d-1} \text{Vol}(S^i) = \prod_{i=1}^{d-1} \frac{\pi^{\frac{i}{2}}}{\Ga(\frac{i}{2}+1)}.
\end{equation}
\vspace{5mm}

\begin{theorem}
Every Siegel set $\Sigma_{t,u} \in \SL(d.\R)$ has finite Haar measure in $G$ and it follows from Theorem \ref{siegel sets cover the fundamental domain} that the Haar measure defined above is finite. Therefore $SL(d,\Z)$ is a lattice in $\SL(d,\R)$.
\end{theorem}

\subsection{Distribution function associated to sucesssive minima and estimates}

\begin{proposition}
For a rank $d$ unimodular lattice $\Lambda \in \mathcal{L}$, let $\lambda_i(\Lambda)$ denote its $i$-th successive minima ($1 \le i \le d$). For any $\delta>0$, we have $$\mu(\{\Lambda \in \mathcal{L}:\lambda_i(\Lambda)=\de \})=0,$$

where $\mu$ is the Haar measure defined on the space of unimodular lattices.
\end{proposition}

\begin{proof}
Indeed, the set $\{\Lambda \in \mathcal{L}:\lambda_i(\Lambda)=\de \}$ is contained in $$S_{\delta}:=\{\Lambda \in \mathcal{L}:\text{there exists a vector } v\in \Lambda \text{ with } \|v\|=\de \}.$$
Noticing that any unimodular lattice can be written as $g\Z^d$ for some $g\in \SL(d,\R)$ and the local identification between Haar measure on $G=\SL(d,\R)$ and $G/\Gamma =\SL(d,\R)/\SL(d,\Z)$, we shall look at the set
\begin{align*}
    T_{\delta}:= &\{g \in G:\text{there exists a vector } v\in \Z^d \text{ with } \|gv\|=\de \}\\
    = & \cup_{v\in \Z^d} \{g \in G: \|gv\|=\de \}.
\end{align*}
This is a countable union and each member in the union is a submanifold of $G$ with lower dimension and hence of zero Haar measure.
\end{proof}

For $x\ge 0$, it would be interesting to give an estimate for the distribution function 
\begin{align*}
    \Phi_i(\de):= \mu(\{\Lambda \in \mathcal{L}: \lambda_i(\Lambda) < \de \})= \mu(\{\Lambda \in \mathcal{L}: \lambda_i(\Lambda) \le \de \})
\end{align*}

For $i=1$, Kleinbock and Margulis gave both lower and upper bounds for $\Phi_1(x)$ \cite{Kleinbock1999LogarithmLF} using a generalized Siegel's formula:

\begin{theorem}[\cite{Kleinbock1999LogarithmLF}, Proposition 7.1]
There exists $C_d, C_d'$ such that 
\begin{equation*}
  C_d \de^d - C_d' \de^{2d}  \le \Phi_1(\de) \le C_d \de^d,
\end{equation*}
for $\de << 1$.
\end{theorem}

The main result in this second is a generalization of the above result to $\la_i$:

\begin{theorem}\label{di distance linear}
For $1\le i < d$, there exists $C_d$ and $C_d'$ such that 
\begin{equation*}
  C_d \de^{di} - o(\de^{di})  \le \Phi_i(\de):=\mu(\{\Lambda \in \mathcal{L}: \lambda_i(\Lambda) \le \de \}) \le C_d' \de^{di},
\end{equation*}
for all $\de << 1$.
\end{theorem}

\begin{corollary}\label{non-dynamical logarithm law}
    For $1 \le i\le d-1$, we have
    \begin{equation*}
       \lim_{t\to \infty}\frac{-\log\mu \left(\left\{\Lambda \in \mathcal{L}: \lambda_i(\Lambda) \le e^{-t} \right\}\right)}{t} = di
    \end{equation*}
\end{corollary}
\begin{proof}
Take $\de=e^{-t}$.
\end{proof}

\vspace{5mm}
For the proof we will use a generalized version of Siegel's mean value formula in geometry of numbers: For a lattice $\Lambda$ in $\R^d$, let $P(\Lambda)$ denote the set of primitive vectors in $\Lambda$, i.e. those vectors that are not a proper integer multiple of any other element in $\Lambda$. Given a real-valued function $f$ on $\R^d$, we define a function $\hat{f}$ on the homogeneous space $X=G/\Gamma:=\SL(d,\R)/\SL(d,\Z)$ by 
$$\hat{f}:=\sum_{v\in P(\Lambda)} f(v)$$

\begin{theorem}[Classical Siegel's Formula \cite{SI45}]
For any $f \in L^1(\mathbb R^d)$, one has 
$$\int_{X} \hat{f} d\mu = c_d \int_{\R^d} f dv,$$
where $c_d=\frac{1}{\zeta(d)}:=\frac{1}{\sum_{n=1}^{\infty} \frac{1}{n^d}}$.
\end{theorem}

Below is a generalization of classical Siegel's formula. First let us recall the notion of primitive tuple from geometry of numbers:

\begin{definition}
For $1\le k \le d$, we say that an ordered $k$-tuple of vectors $(v_1,\dots,v_k) \in 
\underbrace{\Lambda \times \cdots \times  \Lambda}_{d\text{-times}}$ for a lattice $\Lambda \subset \R^d$ is primitive if it is extendable to a basis of $\Lambda$, and denote by $P^k(\Lambda)$ the set of all such $k$-tuples. Note that $P^1(\Lambda)=P(\Lambda)$ above \footnote{Any primitive vector in a lattice can be extended to a basis of the lattice. This follows from the general fact that any element of a free abelian group which is not divisible by any integer bigger than $1$
can be extended to a basis of the abelian group}.
\end{definition}

Now for a function $f\in \R^{dk}=(\R^d)^k$, we define correspondingly
\begin{equation*}
    \hat{f}^k(\Lambda):=\sum_{(v_1,\dots,v_k)\in P^k(\Lambda)} f(v_1,\dots,v_k).
\end{equation*}
Here the superscript on $\hat{f}^k$ should not be confused with the composition (power) of a function. Then we have a generalized Siegel's Formula for primitive tuples which will be helpful for us to estimate the distribution for $\lambda_i(\Lambda)$.

\begin{theorem}[Generalized Siegel's Formula for primitive tuples]\label{generalizedsiegel}
For $1\le k < d$ and $\phi \in L^1(\R^{dk})$, we have
\begin{equation}
    \int_{X} \hat{f}^k d\mu = c_{k,d} \int_{\R^{dk}} f dv_1 \cdots dv_k,
\end{equation}
where $c_{d,k}=\frac{1}{\zeta(d)\cdots\zeta(d-k+1)}$.
\end{theorem}

\begin{proof}
Let $\{e_1,\dots,e_d\}$ be the canonimcal basis of $\R^d$. For $G=\SL(d,\R)$ and $\Gamma=\SL(d,\Z)$ and the $k$-tuple $(e_1,\dots,e_k)$, be , let 
\begin{align*}
     G_k:&=\{g\in G: g.e_i=e_i,\forall 1 \le i \le k\}, \\
     \Gamma_k:&=\{g\in \Gamma: g.e_i=e_i,\forall 1 \le i \le k\}.
\end{align*}
be the stabilizer subgroup of $(e_1,\dots,e_k)$ in $G$ and $\Gamma$, respectively.

Now consider the subset
$$L:=\{(v_1,\dots,v_k) \in \R^{dk}: v_1,\dots,v_k \text{ are linearly independent vectors in }\R^d \}$$

\begin{claim}
L is open dense in $\R^{dk}$ and in particular $\R^{dk}-L$ is of Lebesgue measure zero.
\end{claim}
\begin{proof}[Proof of claim ]\renewcommand{\qedsymbol}{\ensuremath{\#}}
That $L$ is open follows from the condition that linear independence implies that $[v_1,\dots,v_k]$ is a full-rank matrix (there exists at least one $k\times k$ submatrix with determinant zero). 

To see it is dense, we observe that this is equivalent to proving that the set of full-rank $d\times k$ matrices, denoted $F$, is dense in $M(d\times k,\R)$. Noticing that its complement $F^c$ is contained in some subvariety (of stricly lower dimension)
$$\{A\in M(d\times k,\R): \det(A_{k\times k})=0\},$$
for some $k\times k$ submatrix of $A$. Therefore $F$ must be dense in $\R^{dk}$.
\end{proof}

\begin{claim}
$L$ is equal to the $G$-orbit of the $k$-tuple $(e_1,\dots,e_k)$ in $\R^{dk}$.
\end{claim}
\begin{proof}[Proof of claim ]\renewcommand{\qedsymbol}{\ensuremath{\#}}
Indeed, for any $g\in G=\text{SL}(d,\R)$, the tuple $(g.e_1,\dots,g.e_k)$ corresponds to the first $k$ columns of the matrix $g$. However, any $k$ linearly independent vectors $v_1,\cdots, v_k$ in $\R^d$ ($k<d$) can be completed to a $d\times d$ matrix of determinant $1$ (by adding diagonal entries to the last $d-k$ columns). 
\end{proof}

Now consider the map
\begin{align*}
     \phi_G: G \to L\subset \R^{dk}, g\mapsto (g.e_1,\dots,g.e_k)
\end{align*}
By the Orbit-Stabilizer theorem, we have the identification of homogeneous spaces $\phi_G': G/G_k~\tilde{\longrightarrow}~L$. Since $L$ is open dense in $\R^{dk}$ and the Lebesgue measure on $\R^{dk}$ (viewed as a product (Lebesgue) measure on $\underbrace{\R^{d} \times \cdots \times \R^{d}}_{k\text{-times}}$) is invariant under $G=\SL(d,\R)$. The pull-back of the Lebesgue measure on $\R^{dk}$ gives a (unique up to scalar multiple) $G$-invariant Haar measure $\mu_{G/G_k}$ on $G/G_k$ (uniqueness of Haar measure on $G/G_k$ follows from the unimodularity of $G$).

\begin{claim}
$P^k(\Z^d)$ is equal to the $\Gamma$-orbit of the $k$-tuple $(e_1,\dots,e_k)$ in $\R^{dk}$.
\end{claim}
\begin{proof}[Proof of claim ]\renewcommand{\qedsymbol}{\ensuremath{\#}}
Let $(q_1,\dots,q_k)$ be any $k$-tuple of integer vectors in $\Z^d$ that are extendable to a basis $\{q_1,\dots,q_d \}$ of $\Z^d$, and as a basis we have $\det[q_1 \dots q_d ]=1$ (up to adjusting the sign of the last column $q_d$) and hence $(q_1,\dots,q_k)$ lies in the $\Gamma$-orbit of the $k$-tuple $(e_1,\dots,e_k)$.   

On the other hand, for any $(g.e_1,\dots,g.e_k)$, where $g\in \Gamma=\SL(d,\Z)$, clearly $\{g.e_1,\dots,g.e_k\}$ can be completed to a basis $\{g.e_1 \dots ,g.e_d\}$ of $\Z^d$.
\end{proof}
It follows that the map
\begin{equation*}
     \phi_{\Gamma}: \Gamma \to P^k(\Z^d) \subset \R^{dk}, \gamma \mapsto (\gamma.e_1,\dots,\gamma.e_k)
\end{equation*}
gives the identification of $\Gamma$-homogeneous spaces $\phi_{\Gamma}: \Gamma/\Gamma_k ~\tilde{\longrightarrow}~P^k(\Z^d)$, under which the counting measure on $P^k(\Z^d)$ (clearly $\Gamma$-invariant) can be pulled back to a (unique up to scalar multiples) $\Gamma$-invariant Haar measure on $\Gamma/\Gamma_k$, denoted $\mu_{\Gamma/\Gamma_k}$. Note that the summation over $P^k(\Z^d)$ is equal to the integration with respect to the counting measure $\mu_{\Gamma/\Gamma_k}$ over $\Gamma/\Gamma_k$.

By the Lemma 1.6, Chapter I in \cite{RA72}, the invariant measures on $G/\Gamma$ and $\Gamma/\Gamma_k$ give an invariant measure on $\Gamma/\Gamma_k$ and we have the quotient integral formula
\begin{align*}
    &\int_{G/\Gamma} \int_{\Gamma/\Gamma_k} \varphi(gl\Gamma_k)~ d\mu_{\Gamma/\Gamma_k}(l\Gamma) d\mu(g)\\
    =&\int_{G/\Gamma_k}\varphi(g\Gamma_k)~ d\mu_{G/\Gamma_k}(g\Gamma)\\
    =&\int_{G/G_k} \int_{G_k/\Gamma_k} \varphi(gg_k\Gamma_k)~ d\mu_{G_k/\Gamma_k}(g_k\Gamma_k) d\mu_{G/G_k}(gG_k)
\end{align*}
for any $\varphi \in L^1(G/\Gamma_k)$.

Now for any $f\in L^1(\R^{dk})$, we first identify $f$ with a function in $L^1(G/G_k)$ via $\phi_G$. For this new $f$, we define $\varphi_f$ on $G/\Gamma_k$ by setting it as of constant value on each $G_k$-coset:
$$\varphi_f(g\Gamma_k):=f(g G_k),\forall g\in G.$$
In other words, $\varphi_f(g\Gamma_k)= \varphi_f(h\Gamma_k)$, whenever $h^{-1}g\in G_k$.

It follows that
$$\int_{G/\Gamma_k} \varphi_f(g\Gamma_k) d\mu_{G/\Gamma_k}=\mu_{G_k/\Gamma_k}(G_k/\Gamma_k) \cdot \int_{G/G_k} f(g G_k)d\mu_{G/G_k}(g G_k)< \infty.$$
So $\varphi_f \in L^1(G/\Gamma_k)$.

\begin{claim}
The inner integral on the left hand side is (recall that the integration with respect to counting measure on $\Gamma/\Gamma_d$ is the same as the sum over primitive tuples) 
$$\int_{\Gamma/\Gamma_k} \varphi_f(gl\Gamma_k) d\mu_{\Gamma/\Gamma_k}(l)
=\sum_{(v_1,\dots,v_k)\in P^k(g\Z^d)} f(v_1,\dots,v_k)=:\hat{f}^k(g\Z^d).$$
\end{claim}
\begin{proof}[Proof of claim ]\renewcommand{\qedsymbol}{\ensuremath{\#}}
First recall by our identification
\begin{align*}
    &\sum_{(v_1,\dots,v_k)\in P^k(g\Z^d)} f(v_1,\dots,v_k)\\ =&\sum_{(he_1,\dots,he_k)\in P^k(g\Z^d), h\in G} f(he_1,\dots,he_k)\\
    =&\sum_{(he_1,\dots,he_k)\in P^k(g\Z^d), h\in G} f(hG_k)
    \tag{viewing $f$ as a function on $G/G_k$}\\
    =&\sum_{(g^{-1}he_1,\dots,g^{-1}he_k)\in P^k(\Z^d), h\in G} f(hG_k)\\    
    =&\sum_{(le_1,\dots,le_k)\in P^k(\Z^d), l\in G} f(glG_k)
    \tag{change of variable $l:=g^{-1}h$}\\ 
    =&\int_{\Gamma/\Gamma_k} f(glG_k)d\mu_{\Gamma/\Gamma_k}(l\Gamma_k)
    \tag{summation to integration w.r.t. counting measure}\\
    =&\int_{\Gamma/\Gamma_k} \varphi_f(gl\Gamma_k)d\mu_{\Gamma/\Gamma_k}(l\Gamma_k)
    \tag{definition of $\varphi_f$}\\
\end{align*}
\end{proof}

Hence this proves 
\begin{equation*}
    \int_{X} \hat{f}^k d\mu = c_{k,d} \int_{\R^{dk}} f dv_1 \cdots dv_k,
\end{equation*}
with $c_{k,d}=\mu_{G_k/\Gamma_k}(G_k/\Gamma_k)$ to be computed in the Appendix A.
\end{proof}

Now we can give the distribution for $\Phi_i(\de)$:

\begin{proof}[Proof of Theorem \ref{di distance linear}]
Let $B$ be the ball centered at $0$ with radius $\de$.

Note that the condition $\lambda_i(\Lambda)<\de$ implies that there are at least $i$ linearly independent vectors in $\Lambda$ lying in the open ball $B$. However, this does not necessarily mean they can be extended to a basis. But thanks to the Theorem \ref{thmA5}, there exists a basis $v_1,\dots,v_d$ of $\Lambda$ such that
$$\|v_1\|=\lambda_1(\Lambda),\|v_2\|_d \asymp_d \lambda_2(\Lambda),\dots,\|v_d\| \asymp_d \lambda_d(\Lambda).$$

It follows that there exists a constant factor $\eta_d>1$ such that if we dilate the ball $B$ by $\eta_d$ to a new ball $B'$ (centered at the origin with radius $\eta_d \de$), we have
$\Lambda \cap B'$ contains $i$ linearly independent vectors that can be extended to a basis of $\Lambda$. It follows that (since by symmetry $v$ and $-v$ must be contained in $\Lambda \cap B'$ simultaneously and any permutation of this $i$-tuple also gives a new primitive $i$-tuple):
$$|P^i (\Lambda) \cap B'^i|\ge 2^i i!,$$
where $B'^i:=\underbrace{B' \times \cdots \times B'}_{i\text{-times}}\subset \R^{di}$.

\vspace{5mm}
Now take $f:=\mathbf{1}_{B'^{i}}$ and $\hat{f}=    \hat{f}^i(\Lambda):=\sum_{(v_1,\dots,v_i)\in P^i(\Lambda)} f(v_1,\dots,v_i)$ counts the number of points falling into $B'^i$. The left hand side of the generalized Siegel's formula (Theorem \ref{generalizedsiegel}) yields
$$\int_X \hat{f}^i d\mu \ge \int_{\{\Lambda:\lambda_i(\Lambda) \le \de \}}\hat{f}^i d\mu \ge 2^i i!\mu(\{\lambda:\lambda_i(\Lambda)\le \de \}).$$

\vspace{5mm}
On the other hand
$$\int_{\R^{di}}f dv_1 \cdots dv_i=\text{Vol}(B')^i=\eta_d^i \de^{di} c_d^i,$$
where $c_d$ is the volume of unit ball in $\R^d$.

Hence we have the upper bound
$$\mu(\{\Lambda:\lambda_i(\Lambda)\le \de \}) \le \frac{1}{i!}\left(\frac{\eta_d c_d}{2}\right )^i\de^{di}$$

For the lower bound, for $1\le i <d-1$ and $x>0$, let $N(i,x)$ denote the quantity
$$\min \{ |P^i(\Lambda)\cap B(0,x)^i|:\Lambda \in \mathcal{L}, \Lambda \cap B(0,x) \text{~contains at least~} i+1 \text{~linearly independent vectors}\},$$

namely the miminum of the number of all primitive $i$-tuples $(v_1,\cdots, v_i)$ with each component taken from the lattice $\Lambda \cap B(0,x)$ for all unimodular lattice $\La$, given $\Lambda \cap B(0,x)$ contains at least $i+1$ linearly independent vectors. Note that by our assumption and the discussion above, $N(i,\eta_d \de) \ge 2^i i!$. 

Since one can always choose a unimodular lattice with the first $i+1$ sucessive minima small enough to be contained in $B(0,x)$ for any $x>0$ whenever $i<d-1$, we have

\begin{claim}
$N(i,x) \le 2^i(i+1)!$ (independent of $x\in (0,1)$) whenever $i<d-1$ 
\end{claim}

\begin{proof}[Proof of claim ]\renewcommand{\qedsymbol}{\ensuremath{\#}}

Consider the unimodular lattice

$$\Lambda_{x}:=\left \{\frac{3}{4}x e_1,...\frac{3}{4}x e_{i+1},\frac{1}{(\frac{3}{4}x)^{\frac{i+1}{d-i-1}}}e_{i+2},...\frac{1}{(\frac{3}{4}x)^{\frac{i+1}{d-i-1}}} e_{d} \right\}.$$

Note that when $x<1$, we have
$$\La_x \cap B(0,x)=\left\{\pm \frac{3}{4}x e_1,...,\pm \frac{3}{4}x e_{i+1}  \right\}$$
since any integer linear combination
$$n_1 \frac{3}{4}x e_1+...+n_{i+1} \frac{3}{4}x e_{i+1}$$
with some $|n_j|\ge 2$ or at least two of $n_j\ne 0$ must be outside of $B(0,x)$. So $$|P^i(\Lambda_x)\cap B(0,x)^i|\le 2^i i! {i+1 \choose i}=2^i (i+1)!.$$
Therefore $N(i,x)\le 2^i (i+1)!$.
\vspace{5mm}
%On the other hand, if $\Lambda \cap B(0,x)$ contains one family of $i+1$ linearly independent vectors, then there must be at least $2^i i!{i+1 \choose i}$ many primitive $i$-tuples.
\end{proof}

Now set $x=\de$. The idea is to separated the integration domain into two parts: $\{\Lambda:\hat{f}^i(\Lambda)< N(i,\de)\}$ and $\{\Lambda:\hat{f}^i(\Lambda)\ge N(i,\de)\}$. We will see that the integration over the second domain contribute insignificantly as $\de \to 0$. Hence,
\begin{align*}
    \int_X \hat{f}^id\mu 
    = & \int_{\{\Lambda:\hat{f}^i(\Lambda) < N(i,\de)\}} \hat{f}^id\mu + \int_{\{\Lambda:\hat{f}^i(\Lambda)\ge N(i,\de)\}} \hat{f}^id\mu \\
\end{align*}
Notice that by our choice of $B'$, $f$ and the definition of $\hat{f}^i$, $f^i(\La)=|P^i(\La)\cap B^i|$, and the first term \footnote{note that if $i=d-1$ this argument won't make sense since $N(d,\de)$ will become zero if $\de \to 0$ by the Minkowski's second convex body theorem \ref{Min2}.}
\begin{align*}
  \int_{\{\Lambda:\hat{f}^i(\Lambda) < N(i,\de)\}} \hat{f}^id\mu  
  =& \int_{\{\Lambda:\hat{f}^i(\Lambda)< N(i,\de)\}} \hat{f}^id\mu \\
  =& \int_{\{\La:\Lambda \cap B \text{~contains no~} i+1 \text{~linearly independent vectors}, ~\hat{f}^i(\Lambda)< N(i,\de)\}}\hat{f}^i d\mu\\
  \le & \int_{L_i}2^i(i+1)! d\mu\\
  =& 2^i (i+1)!\mu(S_i) \tag{by the estimate from the claim above}
\end{align*}

\iffalse
Now let us find an upper bound for $\hat{f}^i(\Lambda)$ under this constraint. We have

\begin{claim}
$\hat{f}^i(\Lambda)\le 2^i i!$ over the set
$$\{\La:\Lambda \cap B \text{~contains no~} i+1 \text{~linearly independent vectors}, ~\hat{f}^i(\Lambda)< N(i,\de)\}.$$
\end{claim}

\begin{proof}[Proof of Claim]\renewcommand{\qedsymbol}{\ensuremath{\#}}

Indeed, if $\hat{f}^i(\Lambda)> 2^i i!$, then even modulo $\pm$ for each vector, there must be at least two family of primitive $i$-set of vectors in $\Lambda \cap B$. But this in particular implies that there must be at least $i+1$ vectors $v_1,...v_{i+1}$ living in $\Lambda \cap B$, with any $i$ of them 

\end{proof}
\fi

where $L_i$ denote the set of unimodular lattices that contain no $i+1$ linearly independent vectors but contain at least one family of primitive $i$-set of vectors (so that the integral will not vanish). But clearly $S_i\subset \{\La:\la_i(\La)\le \de\}$.

\vspace{5mm}
Now we look at the second term $ \int_{\{\Lambda:\hat{f}^i(\Lambda)\ge N(i,\de)\}} \hat{f}^id\mu$. If $\hat{f}^i(\Lambda)\ge N(i,\de)$, by definition it means the ball $B=B(0,\de)$ contains at least $i+1$ linearly independent vectors in $\La$. But again by symmetry that extra vector has to come in pairs namely $B\cap \Lambda$ has to contain both $v_{i+1}$ and $-v_{i+1}$.

Therefore for such $\Lambda$,
$$|P^i(\Lambda) \cap B^i)| \le \frac{1}{2}|P^{i+1}(\La) \cap B^{i+1})|.$$

Notice that the left hand side is precisely $\hat{f}^i(\La)$. Let $f_1=\mathbf{1}_{B^{i+1}}$, a function in $\R^{d(i+1)}$, we have
\begin{align*}
    \int_{\{\Lambda:\hat{f}^i(\Lambda)
    \ge N(i,\de)\}} \hat{f}^id\mu = &
     \frac{1}{2}\int_{\{\Lambda:\hat{f}^i(\Lambda)> N(i,\de)\}} \hat{f_1}^{i+1}d\mu \\
     \le & \frac{1}{2}\int_X \hat{f_1}^{i+1} d\mu \\
     = & \frac{1}{2}\int_{\R^{d(i+1)}} f_1 dv_1\cdots dv_{i+1} \tag{by the generalized Siegel's formula}\\ 
     = & \frac{1}{2}(\int_{\R^{d}} \mathbf{1}_{B} dv)^{i+1} \\
     = & \frac{1}{2} (c_d \de^d)^{i+1}
\end{align*}

Therefore we obtain the lower bound
\begin{align*}
    \mu(\{\Lambda: \lambda_i(\Lambda) \le \de \})\ge &
    \frac{1}{N(i,\de)} \left(\int_X \hat{f}^id\mu - \int_{\{\Lambda:\hat{f}^i(\Lambda)\ge N(i,\de)\}} \hat{f}^id\mu \right) \\
    = & \frac{1}{2^i (i+1)!} \left(\int_{\R^{di}} f dv_1\cdots dv_i - \int_{\{\Lambda:\hat{f}^i(\Lambda)\ge N(i,\de)\}} \hat{f}^id\mu \right)\\
    = & \frac{1}{2^i(i+1)!}(c_d^i \de^{di}- \frac{1}{2} c_d^{i+1} \de^{d(i+1)}).
\end{align*}

This finishes the case when $i<d-1$.

\vspace{5mm}
To cover the remaining case when $i=d-1$, we will study the following example:

\begin{example}[Measure of a subset of $\{\La:\la_i(\La)\le \de\}$]~\\
In view of Iwasawa decomposition of $G=\SL(d,\R)$, we shall first construct a subset of $G$ that will shrink the first $i$ canonical basis vectors in $\Z^d$: Let $S_i$ denote the collection of all elemnents of the form $kan$, where $ k\in SO(d,\R)$, 
\begin{equation}
 a=\text{diag}(a_1,...a_i,a_{i+1},...,a_d) \in \SL(d,\R), 
\end{equation}
with $\frac{\sqrt 3}{2}a_{j-1}\le a_j<\frac{\de}{\sqrt i}, \forall 1\le j \le i$ (assuming $a_0=0$ as convention) and $\frac{\sqrt 3}{2} a_{j-1}\le a_j \le 1$ whenever $j>i$,
and
\begin{equation}
    n= 
\begin{bmatrix}
1 & n_{12} & n_{13} & ... & n_{1d} \\
0 & 1 & n_{23} & ... & n_{2d} \\
0 & 0 & 1 & ... & n_{3d} \\
\vdots &  \vdots & \vdots &\ddots & \vdots\\
0 & 0 & 0 & ... & 1 \\
\end{bmatrix}
\end{equation}
with $\frac{1}{2} \le n_{ij}\le 1$ for all $i<j$. 

Clearly, with restrictions $\frac{\sqrt 3}{2}a_{j-1}\le a_j$ and $\frac{1}{2}\le n_{ij}$, $S_i$ is contained in the Siegel domain $\Sigma:=\Sigma_{\frac{2}{\sqrt{3}},\frac{1}{2}}$. The reason that we choose to restrict our set to Siegel domain is that over Siegel domain, we have very good control on the overlaps modulo the $\Gamma=\SL(n,\Z)$ action thanks to Theorem \ref{finiteness of nonempty intersections}.

\begin{claim}
$\la_j(S_i\Z^d)\le \de$ for all $j\le i$. 
\end{claim}

\begin{proof}[Proof of Claim ]\renewcommand{\qedsymbol}{\ensuremath{\#}}
Indeed, for $kan \in S_i$ and $j\le i$, 
\begin{align*}
    \|kane_j \|_2
    =&\|[a_1n_{1j},...,a_{i-1}n_{i-1,j},a_i]^T\|_2 \tag{$k$ preserves the distance}\\
    =& \sqrt{a_1^2n_{1j}^2+\cdots+a_{i-1}^2n_{i-1,j}^2+a_i^i}\\
    \le & \sqrt{i\cdot \frac{\de^2}{i}}=\de
\end{align*}
\end{proof}

Therefore $\pi(S_i)\subset \{\La:\la_i(\La)\le \de\}$. Now we shall give a lower bound estimate for the measure of $\pi(S_i)$ in $G/\Gamma$. 

Let $f=\mathbf{1}_{S_i}$ denote the indicator function of $\pi(S_i)$ on $G/\Gamma$ and let $N_d$ denote the (finite) number of $\gamma$ for which $\Sigma \cap F\gamma$ is nonempty, then since $S_i$ is a subset of the Siegel set $\Sigma$, $S_i= \cup_{\gamma \in \Gamma} (S_i \gamma^{-1} \cap F)$ is a finite union of no more than $N_d$ nonempty sets. Let $m_d$ denotes the largest measure of these $N_d$ sets, it follows that
\begin{align*}
    \int_G f(g) dg
    =& \int_{G/\Gamma}\sum_{\gamma \in \Gamma}f(g\gamma)d(g\Gamma)\\
    \le & N_d m_d \\
    \le & N_d \mu_{G/\Gamma}(S_i).
\end{align*}
Now we compute $\int_G f(g) dg$ via Iwasawa decomposition in view of Theorem \ref{decomposition of haar measure on G}:
\begin{align*}
     &\int_G f(g) dg \\
    =& \int_{K}\int_{A}\int_{N} f(kan)\rho(a)dk da dn\\
    =& \int_{K}dk  \int_0^{\de/ \sqrt{i}} \int_{\frac{\sqrt 3}{2}a_1}^{\de/ \sqrt{i}} \cdots \int_{\frac{\sqrt 3}{2}a_{i-1}}^{\de/ \sqrt{i}} \int_{\frac{\sqrt{3}}{2}a_i}^{1} \cdots \int_{\frac{\sqrt{3}}{2}a_{d-1}}^{1}    \frac{\rho(a)}{a_1\dots a_{d-1}} da_{d-1} \dots da_1 \int_{[\frac{1}{2},1]^{d(d-1)}}\prod_{i<j}dn_{ij} \\
    \ge & \text{Vol}(K) \frac{1}{2^{d(d-1)}} \int_0^{\de/ \sqrt{i}} \int_{\frac{\sqrt 3}{2}a_1}^{\de/ \sqrt{i}} \cdots \int_{\frac{\sqrt 3}{2}a_{i-1}}^{\de/ \sqrt{i}} \int_{\frac{[\sqrt 3}{2},1]^{d-i-1}}  \left( a_{d-1}^1 a_{d-2}^3  \cdots a_1^{2d-3} \right) da_{d-1} \dots da_1 \\
    =& \text{Vol}(K) \frac{c_{d,i}}{2^{d(d-1)}} \int_0^{\de/ \sqrt{i}} \int_{\frac{\sqrt 3}{2}a_1}^{\de/ \sqrt{i}} \cdots \int_{\frac{\sqrt 3}{2}a_{i-1}}^{\de/ \sqrt{i}} \left( a_{i}^{2d-1-2i} \cdots a_1^{2d-3} \right) da_{i} \dots da_1 \\ 
    \ge & D_d \de^{(2d-i-1)i}+o(\de^{(2d-i-1)i})
\end{align*}
where the constant $c_d,i$ comes from the integration w.r.t. the variables $a_{i+1},...a_d$ and the exponential $\de^{(2d-i-1)i}$ comes from $(2d-i-1)i=2d-1-2i+\cdots+2d-3+i$ (the last $i$ is from the accumulation of the total order of anti-derivatives of polynomials) and when $i=d-1$, $\de^{(2d-i-1)i}=d(d-1)=di$.

Therefore, for $i=d-1$, we have proved
$$\mu\{\La:\la_i(\La)\le \de\}\ge C_d' \de^{d(d-1)}, \text{~as~} \de<<1.$$
\end{example}
This completes the proof for all $1\le u \le d-1$.
\end{proof}

By looking at the dual lattice, we can also obtain the tail bound for this distribution. 
\begin{corollary}\label{tail bound}
For $1<i\le d$ there exist $C_d$ and $C_d'$ such that 
\begin{equation*}
  C_d \de^{di} \le \mu \left(\left\{\Lambda \in \mathcal{L}: \lambda_i(\Lambda) \ge \frac{1}{\de} \right\}\right) \le C_d' \de^{di},
\end{equation*}
for all $\de << 1$.
\end{corollary}

\begin{proof}
Recall the notion of dual lattice  (cf. \ref{dual lattice}) and notice that the dual map
$$*:X\to X, \Lambda \mapsto \La^*$$
is measure-preserving. By the Theorem \ref{sucessive minima of dual lattice}. We have 
$$1\le \lambda_r(\Lambda) \lambda_{d+1-r} (\Lambda^*) \le d!$$
for any $r=1,2,\cdots d$. Hence the corollary follows.
\end{proof}

\section{Application to the logarithm laws for the flows on the homogeneous space $G/\Ga$.}

Now we define $\Delta_i(\La):=-\log (\la_i(\La))$. It follows from taking the negative logarithm of all sides of the equation \ref{eq:ineq} that $\Delta_i(\La)$ is uniformly continuous. Recall from \cite{Kleinbock1999LogarithmLF}:

\begin{definition}
For a function $\Delta$ on a $G$-homogeneous space $X$, define the tail distribution $\Phi_{\Delta}(z):=\mu\{x\in X:\de(x)\ge z\}$. 

For $k>0$, we will also say that $\de$ is $k$ distance-like if it is uniformly continuous and in addition there exist constants $C_d$ and $C_d'$ such that 
\begin{equation*}
    C_de^{-kz} \le \Phi_{\Delta}(z) \le C_d'e^{-kz}, \forall z\in \R. 
\end{equation*}
\end{definition}

It follows from our Theorem \ref{di distance linear} that $\Delta_i$ is $di$ distance-like in the space of unimodular lattices. And as an immediate consequence of Theorem 1.7 in \cite{Kleinbock1999LogarithmLF}, we have the non-unipotent version of logarithm law:

\begin{theorem}
For any nonzero $(z_1,\dots z_d)$ with $z_1+\dots+z_d=0$, and for almost all unimodular lattice $\Lambda$ in $X=\SL(d,\R)/\SL(d,\Z)$ we have
\begin{equation}
    \limsup_{t\to \infty}\frac{\Delta_i(\exp(tz)\Lambda)}{\log t}=\frac{1}{di}.
\end{equation}
\end{theorem}

For the unipotent flow and first successive minimum, Athreya and Margulis  proved the following logarithm law:

\begin{theorem}[\cite{AM09}, Theorem 2.1]
Let $(u_t)_{t\in \R}$ be a unipotent one-parameter subgroup of $\SL(d,\R)$ and $X:=\SL(d,\R)/\SL(d,\Z)$. For $\mu$-a.e. $\La$ in $X$, we have
\begin{equation*}
    \limsup_{t\to \infty}\frac{-\log \la_1(h_t\La)}{\log t}=\frac{1}{d}.
\end{equation*}
\end{theorem}

We shall generalize this theorem to higher $\la_i$'s:

\begin{theorem}\label{generalize logarithm law for unbounded flow}
Let $(g_t)_{t\in \R}$ be an unbounded one-parameter subgroup of $\SL(d,\R)$ and $X:=\SL(d,\R)/\SL(d,\Z)$. For $\mu$-a.e. $\La$ in $X$, we have
\begin{equation*}
    \limsup_{t\to \infty}\frac{-\log \la_i(h_t\La)}{\log t}=\frac{1}{di}.
\end{equation*}
\end{theorem}

To prove this theorem, we first observe that by Borel-Cantelli Lemma, the upper bound holds for all flows:

\begin{lemma}[The upper bound]
For $\mu$-almost every $\La \in X$, $1\le i \le d-1$, and any one parameter subgroup $(h_t)_{t\in \R}$ of $G=\SL(d,\R)$,
\begin{equation*}
    \limsup_{t\to \infty}\frac{-\log \la_i(h_t\La)}{\log t}\le \frac{1}{di}.
\end{equation*}
\end{lemma}

\begin{proof}
For any $\e>0$, and for $k\ge 1$, let $r_k=(\frac{1}{di}+\e)\log k$ and let $t_k$ be any sequence going to $\infty$ as $k\to \infty$, we have by Theorem \ref{di distance linear} and the fact that $u_{t_k}$ is measure-preserving that
\begin{equation*}
    \mu(\{\Lambda \in X: \lambda_i(u_{t_k}\Lambda) \le e^{-r_k} \}) \le C_d' (e^{r_k})^{di},
\end{equation*}
which is equivalent to
\begin{equation*}
    \mu(\{\Lambda \in X: -\log \la_i(u_{t_k}\La) \ge r_k \}) \le C_d' \frac{1}{k^{1+di\e}}.
\end{equation*}
Since the summatin on the right hand side over $k$ is finite, by Borel-Cantelli Lemma, we have 
\begin{equation*}
    \mu(\limsup_{k\to \infty}\{\Lambda \in X: -\log \la_i(u_{t_k}\La) \ge r_k \})=0.
\end{equation*}
Taking the complement, this means
\begin{equation*}
    \mu(\cup_{N} \cap_{k\ge N}\{\Lambda \in X: -\log \la_i(u_{t_k}\La) < r_k \})=\mu(\liminf_{k\to \infty}\{\Lambda \in X: -\log \la_i(u_{t_k}\La) < r_k \})=1.
\end{equation*}
In other words, for $\mu$-almost every $\La \in X$, there exists $N$ such that $k\ge N$ implies 
\begin{equation*}
    -\log \la_i(u_{t_k}\La) < r_k:=(\frac{1}{di}+\e)\log(k)
\end{equation*}
Since $t_k\to \infty$ is arbitrary, we have
\begin{equation*}
    \limsup_{t\to \infty}\frac{-\log \la_i(h_t\La)}{\log t}\le \frac{1}{di}.
\end{equation*}
for $\mu$-almost all $\La$.
\end{proof}

\vspace{5mm}
To show the lower bound, we shall use a logarithm law for hitting time of unbounded flow against the spherical shrinking target due to Kelmer and Yu \cite{Kelmer2017ShrinkingTP}.

\begin{theorem}[Special Case of Theorem 1.1, \cite{Kelmer2017ShrinkingTP}]\label{spherical log law}
Let $\{B_t\}_{t>0}$ denote a monotone family of spherical (meaning each set $B_t$ is invariant under the left action of $K=\text{SO}(d,\R)$) shrinking (meaning $B_t\supset B_s$ for $t\ge s$ and $\lim_{t\to \infty}\mu(B_t))$ targets in $X := G/\Ga:=\SL(d,\R)/\SL(d,\Z)$. Let $\{g_m\}_{m\in \Z}$
denote an unbounded discrete time flow on $X$ . Then for a.e. $\La \in X$
\begin{equation}
    \lim_{t\to \infty}\frac{\log(\min\{m\in \N:g_m\La\in B_t\})}{-\log(\mu(B_t))}=1
\end{equation}
\end{theorem}

The quantity $\min\{m\in \N:g_m.x\in B_t\}$ is often called the first \textit{hitting time} with respect to the flow $\{g_m\}$ and the target set $B_t$.

\begin{proof}[Proof of \ref{generalize logarithm law for unbounded flow} (the lower bound)]
We will take the shrinking targets as
\begin{equation*}
    B_t:=\{\La: \lambda_i(g_m \La)\le e^{-t}\}, t\ge 0.
\end{equation*}
These sets are clearly spherical, namely $SO(d,\R)$-invariant since $\la_i$. So by Theorem \ref{spherical log law}, 
\begin{equation}
    \lim_{t\to \infty}\frac{\log \min\{m\in \N:\lambda_i(g_m \La)\le e^{-t}\}}{-\log \mu(\{\La: \lambda_i(g_m \La)\le e^{-t}\})}=1.\footnote{The set $\{m\in \N:\lambda_i(g_m \La)\le e^{-t}\}$ is non-empty because $(g_m)$-action is ergodic by Howe-Moore theorem, and thus almost every $(g_m)$-orbit is dense. }
\end{equation}

By Corollary \ref{non-dynamical logarithm law},
    \begin{equation}
       \lim_{t\to \infty}\frac{-\log\mu \left(\left\{\Lambda \in \mathcal{L}: \lambda_i(\Lambda) \le e^{-t} \right\}\right)}{t} = di
    \end{equation}
Therefore by taking the product and reciprocal,
\begin{equation}\label{product of two limits}
     \lim_{t\to \infty}\frac{t}{\log \min\{m\in \N:\lambda_i(g_m \La)\le e^{-t}\}}=\frac{1}{di}.
\end{equation}
But observe that 
\begin{equation}
    \la_i\left( g_{\min\{m\in \N:\lambda_i(g_m \La)\le e^{-t}\}} \La \right) \le e^{-t},
\end{equation}
and therefore 
\begin{equation}\label{hitting time ineq}
   -\log \la_i\left( g_{\min\{m\in \N:\lambda_i(g_m \La)\le e^{-t}\}} \La \right) \ge t,
\end{equation}
It follows that 
\begin{align*}
& \limsup_{t\to \infty}\frac{-\log \la_i(h_t\La)}{\log t}\\
\ge & \limsup_{t\to \infty}\frac{-\log \la_i\left( g_{\min\{m\in \N:\lambda_i(g_m \La)\le e^{-t}\}} \La \right)}{\log \min\{m\in \N:\lambda_i(g_m \La)\le e^{-t}\}} \tag{By the of limsup. The subscript can be considered as a subsequence.}\\
\ge & \limsup_{t\to \infty} \frac{t}{\log \min\{m\in \N:\lambda_i(g_m \La)\le e^{-t}\}}\tag{By the hitting time inequality \ref{hitting time ineq}}\\
= &\frac{1}{di}. \tag{By the limit \ref{product of two limits}}
\end{align*}
%Since vectors representing $\lambda_1(\Lambda),\dots,\lambda_d(\Lambda)$ must be primitive .
This finishes the proof of lower bound and thus the whole logarithm law theorem.
\end{proof}

\begin{appendices}
    \section{Computation of the coefficient $c_{d,k}$ for the generalized Siegel's formula.}
\end{appendices}

The computation of the coefficient in the generalized Siegel's formula requires the Poisson summation formula. Let us first recall the notion of admissible functions and Poisson summation formula from Fourier analysis:

\begin{definition}
A function $f:\R^d \to \R$ is called \textit{admissible} if there exist constants $c_1,c_2>0$ such that both $|f(x)|$ and $|\hat{f}(x)|$ are bounded by $\frac{c_1}{(1+\|x\|)^{d+c_2}}$, where $f\hat{f}(t):=\int_{\R^d}f(x)e^{2\pi i\langle x,t \rangle} dx$ is the Fourier transform of $f$.
\end{definition}

\begin{theorem}[Poisson Summation Formula]
Given any unimodular lattice $\Lambda \in \R^d$, a vector $v$ and an admissible function $f:\R^d\to \R$, we have
$$\sum_{x\in \Lambda}f(x+v)=\sum_{w \in \Lambda^*}e^{-2\pi i \langle v,w \rangle}\hat{f}(t),$$
where $\Lambda^*$ is the dual lattice of $\Lambda$, cf. \ref{dual lattice}.
\end{theorem}

\begin{proposition}
As in the proof of Theorem \ref{generalizedsiegel}, let $\{e_1,\dots,e_d\}$ be the canonimcal basis of $\R^d$. For $G=\SL(d,\R)$ and $\Gamma=\SL(d,\Z)$ and the $k$-tuple $(e_1,\dots,e_k)$, be , let 
\begin{align*}
     G_k:&=\{g\in G: g.e_i=e_i,\forall 1 \le i \le k\}, \\
     \Gamma_k:&=\{g\in \Gamma: g.e_i=e_i,\forall 1 \le i \le k\}.
\end{align*}
be the stabilizer subgroup of $(e_1,\dots,e_k)$ in $G$ and $\Gamma$, respectively. Let $dg$ denote the Haar measure on $G$ (scaled as above) and $dg_k:=d\mu_{G_k}(g_k)$,$d(g\Gamma):=d\mu_{G/\Gamma}(g\Gamma)$, $d\mu_{G/\Gamma_k}(g\Gamma_k)$ denoted the induced Haar measures on $G_k, G/\Gamma$ and $G/\Gamma_k$ respectively. Then,
%If we normalize the Haar measure on $G/\Gamma$ to be $1$, 
$$\mu_{G_k/\Gamma_k}(G_k/\Gamma_k)=\frac{1}{\zeta(d-k+1)\cdots \zeta(d)}$$
\end{proposition}

\begin{proof}~\\
We start from the case when $k=1$. In this case,
\begin{align*}
    G_1:=&\text{Stab}_G\{e_1\}
        =\{g\in \SL(d,\R):ge_1=e_1\}
        =\begin{bmatrix}
            1 & \R^{1 \times (d-1)} \\
            0 & \SL(d-1,\R) 
         \end{bmatrix}\\
         \Gamma_1:=& \text{Stab}_\Gamma\{e_1\}
        =\{g\in \SL(d,\R):ge_1=e_1\}
        =\begin{bmatrix}
                1 & \Z^{1\times (d-1)} \\
                0 & \SL(d-1,\Z) 
              \end{bmatrix}
\end{align*}

For the computation, we first need to recall the Poiss

For $f\in C_c(\R^d)$, namely a countinuous function with compact support function. Futher more, assume that $f$ is $K$-invariant and $f(0)\ne \hat{f}(0):=\int_{\R^d}f(x)dx$. Such function exists. For example, there exists $\eta\in (0,1)$ such that
\begin{equation*}
    f(x)=
    \begin{cases}
        \frac{1-\eta \|x\|}{(1+\|x\|)^{d+1}} & \text{if } x \in B[0,1]\\
        0 & \text{if } x \notin B[0,1]
    \end{cases}
\end{equation*}
satisfy $f(0) \ne \hat{f}(0)$. Other properties are immediate.

\vspace{5mm}
Let 
$\tilde{F}(g):G\to \R$ be defined as 
$$\tilde{F}(g):=\sum_{v\in \Z^d}f(gv).$$

It follows that $\tilde{F}$ is bounded and for any $\gamma\in \Gamma=\SL(d,\Z)$,
$$F(g\gamma)=\sum_{\Z^d}f(gv\gamma)=\sum_{\Z^d}f(gv)=F(g).$$
The $\Gamma$-invariance of $\tilde{F}$ induces a function $F:G/\Gamma \to \R$ by
$$F(g\Gamma):=\sum_{v\in \Z^d}f(gv).$$

\vspace{0.5cm}
Consider the following decomposition of $\Z^d$:
\begin{equation}
    \Z^d=\{0\}\bigsqcup \bigsqcup_{\gamma \Gamma_1 \in \Gamma/\Gamma_1} \bigsqcup_{j=1}^\infty j e_1.
\end{equation}

It follows that $F\in C_c(G/\Ga)$ and that
\begin{align}
      \int_{G/\Gamma}F(g\Gamma)d(g\Gamma)
    =&\int_{G/\Gamma}\sum_{v\in \Z^d}f(gv) d(g\Gamma) \nonumber\\
    =&\int_{G/\Gamma}f(0)d(g\Ga)+\int_{G/\Gamma}\sum_{\ga\Ga_1 \in \Ga/\Ga_1}\sum_{j=1}^{\infty}f(jg\ga e_1) d(g\Ga_1) \nonumber\\
    =&\int_{G/\Gamma}f(0)d(g\Ga)+\int_{G/\Gamma_1}\sum_{j=1}^{\infty}f(jg\ga e_1) d(g\Ga_1) \nonumber\\
    =& f(0)\mu(G/\Ga)+\sum_{j=1}^{\infty} \int_{G/\Gamma_1}f(jg e_1) d(g\Ga_1)
    \label{eq:decomposition of Z^d}
\end{align}

To treat the section part of the sum above, we introduce the following subgroups :
\begin{align*}
    W_1:=&
    \begin{bmatrix}
    1 & 0 \\
    0 & SL(d-1,\R) 
    \end{bmatrix},\\  
    U_1:=&
    \begin{bmatrix}
    1 & \R^{1\times (d-1)} \\
    0 & I_{d-1}
    \end{bmatrix},\\
    A_t:=&\begin{bmatrix}
    t & 0 \\
    0 & t^{-\frac{1}{d-1}}
    \end{bmatrix}.
\end{align*}

The measures on them are canonical ones: $\mu_{W_1}$ is identified with the Haar measure on $\SL(d-1,\R)$, again defined through Iwasawa decomposition above; $\mu_{U_1}$ is identified with the Lebesgue measure on $\R^{d-1}$; and the measure $A_t$ is identified with $\frac{dt}{t}$ on $\R_{>0}$.

Clearly $G_1=W_1U_1$. 

\begin{claim} $t^d \frac{dt}{t} dw du$ defines a right invariant measure on $A_tW_1U_1$.
\end{claim}

\begin{proof}[Proof of Claim]\renewcommand{\qedsymbol}{\ensuremath{\#}}
Indeed, for any continuous function $f$ with compact support defined on $A_tW_1U_1$, identified with $\R_{t>0}\times \SL(d-1,\R) \times \R^{d-1}$ and $a=\text{diag}(t,t^{-\frac{1}{d-1}}I_{d-1})$, $w'\in W_1$ and $u'\in U_1$. Notice that $awa'^{-1}=w$ and the Jacobian of the map $u'\mapsto aua'^{-1}$ is $(t^{1+\frac{1}{d-1}})^{d-1}=t^{d-1}$, the same change of variable argument for the integral 
$$\int_A \int_{W_1} \int_{U_1} f(awua'w'u')dadwdu$$
gives the right invariance of $t^d \frac{dt}{t} dw du$.
\end{proof}

Observe that $$K\cap AW_1U_1=\begin{bmatrix}
1 & 0\\
0 & \text{SO}(d-1,\R)
\end{bmatrix}\cong \text{SO}(d-1,\R)$$ and that the map
$$K\times G_1 \to KG_1, (k,g) \mapsto k^{-1}g$$
has its fiber at the identity equal to $K\cap AW_1U_1\cong \text{SO}(d-1,\R)$,
we have by the quotient integral formula and the proof of Theorem 8.32 in \cite{KN02}, for any compactly supported continuous function $\phi$ on $G$,
$$\int_{G}\phi(jg e_1) d(g\Ga)=
\frac{1}{\text{Vol}(\text{SO}(d-1,\R))}\int_{KA_tW_1U_1}\phi(jkawu e_1) dk t^{d}da dw du.$$
Since any compactly supported function $\Phi$ on $G/\Ga_1$ can be expressed as
$$\Phi(g\Ga_1)=\sum_{\ga \in \Ga_1}\phi(g\ga),$$
for some compactly supported continuous function $\phi$ on $G$, we have by quotient integral formula and the uniqueness of Haar measure on homogeneous space $G/\Gamma_1$
$$\int_{G/\Gamma_1}f(jg e_1) d(g\Ga)=
\int_{KA_tW_1U_1/\Gamma_1}f(jkawu e_1) dk t^{d}da d(wu\Ga_1)$$
where $f\in C_c(\R^d)$ as above and $j\ge 1$. 

Now it follows from \ref{eq:decomposition of Z^d} that
\begin{align*}
    &\int_{G/\Gamma}F(g\Gamma)d(g\Gamma)\\
    =& f(0)\mu(G/\Ga)+\sum_{j=1}^{\infty} \int_{G/\Gamma_1}f(jg e_1) d(g\Ga_1)\\
    =& f(0)\mu(G/\Ga)+ \frac{1}{\text{Vol}(\text{SO}(d-1,\R))}\sum_{j=1}^{\infty}\int_{K}\int_{A_tW_1U_1/\Gamma_1}f(jkawue_1) dk t^{d}da d(wu\Ga_1) \\
    =& f(0)\mu(G/\Ga)+ \frac{\text{Vol}(\text{SO}(d,\R)}{\text{Vol}(\text{SO}(d-1,\R))}\sum_{j=1}^{\infty}\int_{A_tW_1U_1/\Gamma_1}f(jawue_1) t^{d}da d(wu\Ga_1) \tag{$f$ is $K$-invariant by assumption}\\
    =& f(0)\mu(G/\Ga)+ \text{Vol}(S^{d-1})\sum_{j=1}^{\infty}\int_{A_tW_1U_1/\Gamma_1}f(jawue_1) t^{d}da d(wu\Ga_1)\\
    =& f(0)\mu(G/\Ga)+ \text{Vol}(S^{d-1})\sum_{j=1}^{\infty}\int_{A_t}\int_{W_1U_1/\Gamma_1}f(jawue_1) t^{d}da\\ 
    =& f(0)\mu(G/\Ga)+ \text{Vol}(S^{d-1})\sum_{j=1}^{\infty}\int_{A_t}\int_{W_1U_1/\Gamma_1}f(jtwue_1) t^{d-1}dt\\ 
    =& f(0)\mu(G/\Ga)+ \text{Vol}(S^{d-1})\text{Vol}(\SL(d-1,\R)/\SL(d-1,\Z))\sum_{j=1}^{\infty}\int_{0}^{\infty}f(jte_1) t^{d-1}dt
    \tag{$\text{Vol}(\SL(d-1,\R)/\SL(d-1,\Z))=\text{Vol}(W_1U_1/\Ga_1)$ since $\text{Vol}(\R^d/\Z^d)$=1.}\\
    =& f(0)\mu(G/\Ga)+ \text{Vol}(S^{d-1})\text{Vol}(\SL(d-1,\R)/\SL(d-1,\Z)) \sum_{j=1}^{\infty}\frac{1}{j^d}\int_{0}^{\infty}f(te_1) t^{d-1}dt \\
    \tag{change of variable $t\mapsto \frac{t}{j}$}
\end{align*}
Note that the above only works $d>2$. For the case when $d=2$, this $\text{Vol}(\SL(d-1,\R)/\SL(d-1,\Z))$ has to be replaced by $1$.
\begin{claim}
$$\text{Vol}(S^{d-1})\int_{0}^{\infty}f(te_1) t^{d-1}dt=\hat{f}(0),$$
\end{claim}
where $\hat{f}(0)=\int_{\R^d}f(x)dx$ is the Fourier transform of $f$ at $0$. 
\begin{proof}[Proof of Claim]\renewcommand{\qedsymbol}{\ensuremath{\#}}
Indeed, notice that $f$ is $K$-invariant (rotation invariant), so 
$$f(te_1)=f(x), \forall x\in \R^d \text{ with }\|x\|=t.$$
By the spherical coordinate in $\R^d$, we have
\begin{align*}
x_1 &= r \cos(\varphi_1) \\
x_2 &= r \sin(\varphi_1) \cos(\varphi_2) \\
x_3 &= r \sin(\varphi_1) \sin(\varphi_2) \cos(\varphi_3) \\
    &\,\,\,\vdots\\
x_{d-1} &= r \sin(\varphi_1) \cdots \sin(\varphi_{d-2}) \cos(\varphi_{d-1}) \\
x_d     &= r \sin(\varphi_1) \cdots \sin(\varphi_{d-2}) \sin(\varphi_{d-1}) .
\end{align*}
where $0 \le \phi_{d-1}\le 2\pi$ and $0 \le \phi_{i}\le \pi$ for all $i\le d-1$.
and
\begin{align*}
    \int_{\R^d}f(x)dx
    =& \int_{0}^{\infty}\cdots\int_{0}^{\infty}f(x_1,\dots,x_d)dx_1\cdots dx_d\\
    =& \int_{[0,\pi]^{d-1}\times [0,2\pi]}\int_{0}^{\infty} f(x_1,\dots,x_d)\left|\det\frac{\partial (x_i)}{\partial\left(r,\varphi_j\right)}\right|
dr\,d\varphi_1 \, d\varphi_2\cdots d\varphi_{d-1} \\
=& \int_{[0,\pi]^{d-1}\times [0,2\pi]}\int_{0}^{\infty} f(re_1)
r^{d-1}\sin^{d-2}(\varphi_1)\sin^{d-3}(\varphi_2)\cdots \sin(\varphi_{d-2})\,
dr\,d\varphi_1 \, d\varphi_2\cdots d\varphi_{n-1}\\
=& \text{Vol}(S^{d-1})\int_{0}^{\infty} f(re_1)
r^{d-1}dr.
\end{align*}
\end{proof}
Therefore, we have
\begin{align}
    \int_{G/\Gamma}F(g\Gamma)d(g\Gamma)   =& f(0)\mu(G/\Ga)+ \text{Vol}(S^{d-1})\text{Vol}(\SL(d-1,\R)/\SL(d-1,\Z)) \sum_{j=1}^{\infty}\frac{1}{j^d}\int_{0}^{\infty}f(te_1) t^{d-1}dt \nonumber \\
    =& f(0)\mu(G/\Ga)+ \hat{f}(0)\text{Vol}(\SL(d-1,\R)/\SL(d-1,\Z)) \sum_{j=1}^{\infty}\frac{1}{j^d} \nonumber \\
    =& f(0)\mu(G/\Ga)+ \hat{f}(0)\text{Vol}(\SL(d-1,\R)/\SL(d-1,\Z)) \zeta(d). \label{eq:recursive relation}
\end{align}

In other to find $\text{Vol}(\SL(d-1,\R)/\SL(d-1,\Z))$, we shall look at the dual version of the above equation.

For any $g\in G$, $g\Z^d$ defines a lattice and its dual lattice is $g^*\Z^d={}^tg^{-1}\Z^d$ (Proposition \ref{dual is same as transpose inverse}). But the automorphism
$$*:G\to G, g\mapsto g^*$$

clearly preserves the Haar measure on $G$ and $\ga \Z^d=\ga^* \Z^d$ for all $\ga \in \Ga$. So it also preserves the Haar measure on $G/\Ga$.

On the other hand, by the Poisson summation formula:
$$F(g\Gamma)=\sum_{v\in \Z^d}f(gv)=\sum_{v\in \Z^d}\hat{f}(g^*v)=:\hat{F}(g^*).$$

Since $\hat{\hat{f}}(0)=f(0)$, by replacing $f$ in the recursion equation \ref{eq:recursive relation} by $\hat{f}$, we have
\begin{align*}
     & f(0)\mu(G/\Ga)+ \hat{f}(0)\text{Vol}(\SL(d-1,\R)/\SL(d-1,\Z)) \zeta(d)\\
    =& \int_{G/\Gamma}F(g\Gamma)d(g\Gamma)
    =\int_{G/\Gamma}\hat{F}(g\Gamma)d(g\Gamma) \\
    =& \hat{f}(0)\mu(G/\Ga)+ \hat{\hat{f}}(0)\text{Vol}(\SL(d-1,\R)/\SL(d-1,\Z)) \zeta(d)\\
    =& \hat{f}(0)\mu(G/\Ga)+ f(0).
\end{align*}
Since we have chosen $f(0)=\hat{f}(0)$ at the beginning, this yields:
$$\text{Vol}(G/\Ga)=\text{Vol}(\SL(d,\R)/\SL(d,\Z))=\zeta(d)\text{Vol}(\SL(d-1,\R)/\SL(d-1,\Z))=\zeta(d)\text{Vol}(G_1/\Ga_1),$$
for $d>2$ and with our discussion above (before the claim), 
$$\text{Vol}(\SL(2,\R)/\SL(2,\Z))=\zeta(2).$$ 
By induction, we have 
$$\text{Vol}(G/\Ga)=\zeta(d)\cdots \zeta(d-k+1)\text{Vol}(G_k/\Ga_k)$$
This gives the constant we need for the generalized Siegel formula as well as
$$\text{Vol}(G/\Ga)=\zeta(d)\cdots \zeta(2).$$
\end{proof}

\section*{Acknowledgement}
The author would like to thank Professors Michael Bersudsky and Nimish Shah for their encouragement and helpful discussions on this project. 

\printbibliography[
heading=bibintoc,
title={Bibliography}
]

@article{Kelmer2017ShrinkingTP,
  title={Shrinking target problems for flows on homogeneous spaces},
  author={Dubi Kelmer and Shucheng Yu},
  journal={Transactions of the American Mathematical Society},
  year={2017}
}

@inproceedings{Aj98,
author = {Ajtai, Mikl\'{o}s},
title = {The Shortest Vector Problem in L2 is NP-Hard for Randomized Reductions (Extended Abstract)},
year = {1998},
isbn = {0897919629},
publisher = {Association for Computing Machinery},
address = {New York, NY, USA},
url = {https://doi.org/10.1145/276698.276705},
doi = {10.1145/276698.276705},
booktitle = {Proceedings of the Thirtieth Annual ACM Symposium on Theory of Computing},
pages = {10–19},
numpages = {10},
location = {Dallas, Texas, USA},
series = {STOC '98}
}

@article{1981Another,
  title={Another NP-complete partition problem and the complexity of computing short vectors in a lattice},
  author={Boas, Pve},
  year={1981},
}

@article{HELFRICH1985125,
title = {Algorithms to construct minkowski reduced and hermite reduced lattice bases},
journal = {Theoretical Computer Science},
volume = {41},
pages = {125-139},
year = {1985},
issn = {0304-3975},
doi = {https://doi.org/10.1016/0304-3975(85)90067-2},
url = {https://www.sciencedirect.com/science/article/pii/0304397585900672},
author = {Bettina Helfrich}
}

@article{Fo07,
 author = {Gerald B. Folland},
 journal = {Wiley},
 title = {Real Analysis: Modern Techniques and Their Applications, 2nd Edition},
 year = {2007}
}

@article{SI45,
 ISSN = {0003486X},
 URL = {http://www.jstor.org/stable/1969027},
 author = {Carl Ludwig Siegel},
 journal = {Annals of Mathematics},
 number = {2},
 pages = {340--347},
 publisher = {Annals of Mathematics},
 title = {A Mean Value Theorem in Geometry of Numbers},
 urldate = {2022-10-18},
 volume = {46},
 year = {1945}
}

@article{AM09,
title = {Logarithm laws for unipotent flows, I},
journal = {Journal of Modern Dynamics},
volume = {3},
number = {3},
pages = {359-378},
year = {2009},
author = {Jayadev S. Athreya and Gregory A. Margulis},
}

@article{Kleinbock1999LogarithmLF,
  title={Logarithm laws for flows on homogeneous spaces},
  author={Dmitry Kleinbock and G. A. Margulis},
  journal={Inventiones mathematicae},
  year={1999},
  volume={138},
  pages={451-494}
}

@Book{RA72,
 author = {Madabusi S. Raghunathan},
 publisher = {Springer Berlin, Heidelberg},
 title = {Discrete Subgroups of Lie Groups},
 year = {1972}
}

@Book{KN02,
 author = { Anthony W. Knapp},
 publisher = {Springer},
 title = {Lie Groups: Beyond an Introduction, 2nd Edition},
 year = {2002}
}

@Book{CA97,
 author = {J. W. S. Cassels},
 publisher = {Springer},
 title = {An Introduction to the Geometry of Numbers},
 year = {1997}
}

@article{MI1896,
author = {Minkowski,Hermann},
title = {Geometrie der Zahlen},
journal = {Leipzig : Teubner},
year = {1896},
}

@article{BM00,
author = {Bekka, M. Bachir and Mayer, Matthias},
title = {Ergodic Theory and Topological Dynamics of Group Actions on Homogeneous Spaces},
journal = {Cambridge University Press},
year = {2000},
doi = {https://doi.org/10.1007/978-3-319-05792-7},
}

@article{Fo15,
author = {Folland, Gerald B. },
title = {A course in Abstract Harmonic Analysis, 2nd edition},
journal = {Chapman and Hall/CRC},
year = {2015},
doi = {https://doi.org/10.1007/978-3-319-05792-7},
}
\end{document}